\documentclass[12pt]{amsart}
\usepackage{bbm}
\usepackage[pdftex]{graphicx}
\usepackage{amscd, amssymb, dsfont, upgreek, mathrsfs}
\input{xy}
\xyoption{all}

\newtheorem{Thm}{Theorem}
\newtheorem{Conj}[Thm]{Conjecture}
\newtheorem{Prop}[Thm]{Proposition}
\newtheorem{Def}[Thm]{Definition}
\newtheorem{Def/Thm}[Thm]{Definition/Theorem}
\newtheorem{Cor}[Thm]{Corollary}
\newtheorem{Lemma}[Thm]{Lemma}

\theoremstyle{remark}

\newcommand{\ti }{\times}
\newcommand{\ot }{\otimes}

\newcommand{\lann}{\langle\langle}
\newcommand{\rann}{\rangle\rangle}
\newcommand{\lannn}{\left\langle\left\langle}
\newcommand{\rannn}{\right\rangle\right\rangle}

\newcommand{\NN}{{\mathbb N}}

\newcommand{\PP }{{\mathbb P}}

\newcommand{\QQ }{{\mathbb Q}}
\newcommand{\CC }{{\mathbb C}}
\newcommand{\ZZ }{{\mathbb Z}}

\newcommand{\vir}{\mathrm{vir}}

\newcommand{\DD}{\mathsf{D}}

\newcommand{\T}{{\mathsf{T}}}

\newcommand{\lan}{\langle}
\newcommand{\ran}{\rangle}

\newcommand{\pP}{\mathsf{P}}

\newcommand{\ppl}{{\mathsf{P}}\left[}
\newcommand{\ppr}{\right]}

\def \proj {{\mathbb{P}}}
\newcommand{\com}{{\mathbb{C}}}

\begin{document}

\title[Gromov-Witten invariants of Calabi-Yau manifolds with two K\"{a}hler parameters]
{Gromov-Witten invariants of Calabi-Yau manifolds with two K\"{a}hler parameters}

\author{Hyenho Lho}
\address{Department of Mathematics, ETH Z\"urich}
\email {hyenho.lho@math.ethz.ch}
\date{April 2018}.

\begin{abstract} 
We study the Gromov-Witten theory of $K_{\PP^1\ti\PP^1}$ and some Calabi-Yau hypersurface in toric variety. We give a direct geometric proof of the holomorphic anomaly euqation for $K_{\PP^1\ti\PP^1}$ in the form predicted by B-model physics. We also calculate the closed formula of genus one quasimap invariants of Calabi-Yau hypersurface in $\PP^{m-1}\ti\PP^{n-1}$ after restricting second K\"{a}hler parameter to zero. By wall-crossing theorem between Gromov-Witten and quasimap invariants, we can obtain the genus one Gromov-Witten invariants.
 \end{abstract}

\maketitle

\setcounter{tocdepth}{1} 
\tableofcontents

\setcounter{section}{-1}

\section{Introduction}

While the Gromov-Witten theory of Calabi-Yau manifold with one K\"{a}hler parameter is well studied (\cite{GJR,KL,Popa1,Zg1}), the Gromov-Witten theory of manifold with more than one K\"{a}hler parameter is not well studied yet. In this paper we study the Gromov-Witten invariants of several manifolds with two K\"{a}hler parameters.
 
 In \cite{CKg0,CKg, CKw,CKM} the quasimap invariants were introduced and its relationship between the Gromov-Witten invariants are also studied. The quasimap theory is often more effective in computational or theoretical aspects. Since for all examples in our paper, the Gromov-Witten theory is equivalent to the quasimap theory by the results of \cite{CKg,CKw,CJR}, we study the quasimap theory instead of the Gromov-Witten theory.

\subsection{Toric Calabi-Yau manifolds}
The computation of the Gromov-Witten invariant of toric Calabi-Yau manifolds is relatively more accessible due to torus localization theorem. For example the Gromov-Witten invariants of $\mathcal{O}_{\PP^1}(-1)\oplus\mathcal{O}_{\PP^1}(-1)$ are completely solved in \cite{GP} for all genus. For $K_{\PP^2}$, the Gromov-Witten theory is studied for all genus in terms of holomorphic anomaly equation in \cite{KL}. In Section 3 we consider the Gromov-Witten theory of total space of the canonical bundle $K_{\PP^1\ti\PP^1}$ over $\PP^1 \times \PP^1$. We give a direct geometric approach using the quasimap invariants of $K_{\PP^1\ti\PP^1}$. Another approach using the topological recursion of Eynard and Orantin might be possible (\cite{BKMP,EMO,FLZ}). 

Define genus $g$ quasimap potential function of $K_{\PP^1\ti\PP^1}$

$$\mathsf{F}_g(q_1,q_2):=\sum_{d_1,d_2 \ge 0}q_1^{d_1}q_2^{d_2}\int_{[\overline{Q}_{g,0}(K_{\PP^1\ti\PP^1},(d_1,d_2))]^{vir}}1.$$
We study the holomorphic anomaly equations for the new series
$$\overline{\mathsf{F}}_g(q):=\mathsf{F}_g(q,q).$$
In order to state the holomorphic anomaly equations, we require the following power series{\footnote{The relationship between the series $L$, $C_1$, $A_2$ and the generators of the ring of quasi modular forms with respect to the groups $\Gamma_0(4)$ is explained in \cite{ASYZ}}} in $q$.
\begin{align*}
L(q)&=(1-16q)^{-\frac{1}{4}}\\
I_1(q)&=\sum \frac{2(2d)!(2d-1)!}{(d!)^4}q^d\\
C_1(q)&=1+q \frac{\partial}{\partial q} I_1\\
A_2(q)&=\frac{1}{L^4}\left(\frac{q\frac{d}{dq}C_1}{C_1}-\frac{L^4}{4}+\frac{1}{2}\right)
\end{align*}

The ring $\mathbb{C}[L^{\pm1}]=\mathbb{C}[L,L^{-1}]$ will play a basic role.
Consider the
free polynomial rings in the variables $A_2$ and $C_1^{-1}$ over $\com[L^{\pm1}]$,
\begin{equation}
\label{dsdsds7}
\mathbb{C}[L^{\pm1}][A_2]\, , \ \ \ 
\mathbb{C}[L^{\pm1}][A_2,C_1^{-1}]\,.
\end{equation}
There are canonical maps
\begin{equation}\label{dsdsds22}
\mathbb{C}[L^{\pm1}][A_2]\rightarrow \mathbb{C}[[q]]\, , \ \ \ 
\mathbb{C}[L^{\pm1}][A_2,C_1^{-1}]\rightarrow \mathbb{C}[[q]]
\end{equation}
given by assigning the above defined series $A_2(q)$ and
$C_1^{-1}(q)$ to the variables $A_2$ and $C_1^{-1}$ respectively.
We may therefore consider elements of the rings \eqref{dsdsds7} {\em either} 
as free polynomials in the variables $A_2$ and $C_{1}^{-1}$ {\em or}
as series in $q$.

Let $F(q)\in \mathbb{C}[[q]]$ be a series in $q$. When we write
$$F(q) \in \mathbb{C}[L^{\pm1}][A_2]\, ,$$
we mean there is a canonical lift ${F}\in \mathbb{C}[L^{\pm1}][A_2]$
for which
$${F} \mapsto F(q) \in \mathbb{C}[[q]]$$
under the map \eqref{dsdsds22}. 
The symbol $F$ {\em without the
argument $q$} is the lift.
The notation 
$$F(q) \in \mathbb{C}[L^{\pm1}][A_2,C_1^{-1}]$$
is parallel.

Let $T$ be the standard coordinate mirror to $t=\text{log}(q)$,
$$T=\text{log}(q)+I_1(q)\,.$$
Then $Q(q)=\text{exp}(T)$ is the usual mirror map.

\begin{Thm} \label{Th1}For the quasimap invariants of $K_{\PP^1\times\PP^1}$,
\begin{enumerate}
 \item [(i)] $\overline{\mathsf{F}}_g(q) \in \CC[L^{\pm 1}][A_2]$ for $g \ge 2$,
 \item [(ii)] $\frac{\partial^k \overline{\mathsf{F}}_g}{\partial T^k}(q) \in \CC[L^{\pm 1}][A_2,C_1^{-1}]$ for $g \ge 1$ and $k \ge 1$.
\end{enumerate}

\end{Thm}

\begin{Thm}\label{Th2}The hololorphic anomaly equations for the quasimap invariants of $K_{\PP^1\times\PP^1}$ hold for $g \ge 2$:
\begin{align*}
\frac{1}{C_1^2}\frac{\partial \overline{\mathsf{F}}_g}{\partial A_2}=\frac{1}{2}\sum_{i=1}^{g-1}\frac{\partial \overline{\mathsf{F}}_{g-i}}{\partial T}\frac{\partial \overline{\mathsf{F}}_{i}}{\partial T}+\frac{1}{2}\frac{\partial^2\overline{\mathsf{F}}_{g-1}}{\partial T^2}\,.
\end{align*}
\end{Thm}

 The derivative of $\overline{\mathsf{F}}_g$ with respect to $A_2$ in the equation of Theorem \ref{Th2} is well-defined by (i) in Theorem \ref{Th1}. Theorem \ref{Th2} determine $\overline{\mathsf{F}}_g\in\CC[L^{\pm 1}]$ uniquely as a polynomial in $A_2$ up to a constant term in $\CC[L^{\pm 1}]$. The degree of the constant term is bounded as will be seen in the proof of Theorem \ref{Th2}. So Theorem \ref{Th2} determine $\overline{\mathsf{F}}_g$ recursively from the lower genus theory together with a finite amount of data.

 \subsection{Calabi-Yau hypersurface in toric manifold.}
 Genus one Gromov-Witten invariants are non-trivial even for Calabi-Yau manifold with dimension not necessarily equal to $3$. Dimension $4$ and $5$ cases were studied in \cite{KP,PZ}. In Section 4 we study the genus one Gromov-Witten theory of Calabi-Yau  hypersurface in toric variety. 
 
 Let $X$ be a zero loci of generic section of the anti-canonical bundle of
 $$\PP^1\times P,$$
 where $P$ is smooth Fano toric variety. As before, denote by 
 $$\mathsf{F}_1:=\sum_{\beta\in H^2(X,\QQ)}q^{\beta}\int_{[\overline{Q}_{1,0}(X,\beta)]^{\text{vir}}}1\,$$
the genus one quasimap potential function of $X$.

\begin{Conj}Let $X$ be a zero loci of general section of the anti-canonical bundle of $\PP^1\ti P$ for smooth Fano toric variety $P$.
\begin{align*}
q\frac{d}{dq}\mathsf{F}_1(q,0,\dots,0)=N^X_{1,(1,\mathbf{0})}\frac{q}{1-4q}.
\end{align*}
where $N^X_{1,(1,\mathbf{0})}$ is genus $1$, degree $(1,\mathbf{0})$ quasimap invariant of $X$. 
\end{Conj}
We prove the conjecture for $P$ equal to
 \begin{align*}
 \PP^{n-1} \,\, \text{for}\, n\ge 2.
 \end{align*}

\noindent More generally, we prove the following theorem in Section 4.

\begin{Thm} \label{Th4}
 Let $X$ be a zero loci of general section of anti-canonical bundle of $\PP^{m-1}\ti\PP^{n-1}$. Then we have the following formula for genus one quasimap potential function of $X$.
 \begin{multline*}q\frac{d}{dq}\mathsf{F}_1(q,0)=\frac{n}{48}(3m^2-11m-n^2+n+8)\frac{m^m q}{1-m^mq}\\-\frac{n}{2}\sum_{k=0}^{m-3}\binom{m-1-k}{2}\frac{q\frac{d}{dq}C_k}{C_k}\,.
 \end{multline*}
\end{Thm}

\noindent Here $C_k(q)$ is some series in $q$. See \eqref{Ckg} for the definition of $C_k(q)$.

\subsection{Plan of the paper.} After a review of the moduli space of quasimaps in Section 1 and the genus zero theory in Section 2, Theorem 1 and 2 are proven in Section 3 using torus localization formula for $K_{\PP^1\ti\PP^1}$ combined with quantum Lefschetz hyperplane section theorem. Theorem 4 is proven in Section 4 applying the result of \cite{KL}.

\subsection{Notations.}\label{No} For the series which are defined in Section 2, we will use the following conventions.

Let $\mathbf{Z}\in\CC[[t_1,t_2,q_1,q_2]]$ be the series in $t_1,t_2,q_1,q_2$.
We define a new series $\mathsf{Z},\overline{\mathsf{Z}}$ and $Z$ by

\begin{align}\label{notations}
  \nonumber  \mathsf{Z}=&\mathbf{Z}|_{t_1=0,t_2=0}\in\CC[[q_1,q_2]]\,,\\
    \overline{\mathsf{Z}}=&\mathsf{Z}|_{q_1=q_2=q}\in\CC[[q]]\,,\\
   \nonumber Z=&\mathsf{Z}|_{q_1=q,q_2=0}\in\CC[[q]]\,.
\end{align}

\subsection{Acknowledgments} 
I am very grateful to  I.~Ciocan-Fontanine, H.~Iritani, B.~Kim, A.~Klemm, M. C.-C. Liu, R. ~Pandharipande, E.~Scheidegger, J. ~Shen and Z. Zong
for discussions over the years 
about the moduli space of quasimap and the invariants of Calabi-Yau geometries. I was supported by the grant ERC-2012-AdG-320368-MCSK.

\
\

\section{Quasimap invariants and B-model}

Let $\tilde{X}$ be a affine algebraic variety with an action by a reductive algebraic group $\mathsf{G}$. Denote by X the GIT quotient
$$\tilde{X}/\!\!/_{\!\theta}\mathsf{G}$$
with some choice of a character $\theta$ of $\mathsf{G}$. Let $L_\theta$ be the $\mathsf{G}$-linearized line bundle on $\tilde{X}$ determined by $\theta$. Let
$$\tilde{X}^s=\tilde{X}^s(\theta)\,\,\, \text{and}\,\,\, \tilde{X}^{ss}=\tilde{X}^{ss}(\theta)$$
be the open subsets of stable (resp. semistable) points with respect to $\theta$. We assume
\begin{itemize}
    \item [(i)] $\tilde{X}^s=\tilde{X}^{ss}$\,,
    \item [(ii)] $\tilde{X}^s$ is nonsingular,
    \item [(iii)] $\mathsf{G}$ acts freely on $\tilde{X}^s$.
\end{itemize}

Fix integers $n,g\ge0$ and a class $\beta\in \text{Hom}_\ZZ(\text{Pic}^\mathsf{G}(\tilde{X}),\ZZ)$. With this setup, there are quasimap moduli space

$$\overline{Q}_{g,n}(X,\beta)$$
with the canonical virtual fundamental class $[\overline{Q}_{g,n}(X,\beta)]^{vir}$ (\cite{CKM}). The moduli space $\overline{Q}_{g,n}(X,\beta)$ parameterize the data
$$(\,\mathcal{C},\{p_1,p_2\,\dots,p_n\},\mathcal{P},u\,)$$
where
\begin{itemize}
    \item [$\bullet$] $(\mathcal{C},\{p_1,p_2,\dots,p_n\})$ is a connected, at most nodal, $n$-pointed projective curve of genus $g$,
    \item [$\bullet$] $\mathcal{P}$ is a principal $\mathsf{G}$-bundle on $\mathcal{C}$,
    \item [$\bullet$] $u$ is a section of the induced fiber bundle $\mathcal{P}\ti_{\mathsf{G}}\tilde{X}$ with fiber $\tilde{X}$ on $\mathcal{C}$ such that $(\mathcal{P},u)$ is of class $\beta$\,,
\end{itemize}
satisfying the following generic nondegeneracy condition and stability condition:

\begin{itemize}
    \item [(i)] There is a finite (possibly empty) set $\mathcal{B}\subset \mathcal{C}$ such that for every $p\in\mathcal{C} \setminus \mathcal{B}$ we have $u(p)\in \tilde{X}^{s}$,
    \item [(ii)] The set $\mathcal{B}$ is disjoint from the set of the nodes and markings on $\mathcal{C}$,
    \item [(iii)] The line bundle $\omega_{\mathcal{C}}(\sum_{i=1}^n p_i)\ot\mathcal{L}^\epsilon$ is ample for every rational number $\epsilon >0$, where
    $$\mathcal{L}:=u^*(\mathcal{P}\ti_{\mathsf{G}}L_\theta)\,.$$
\end{itemize}

\

For Calabi-Yau 3-fold $X$, define the genus $g$ quasimap potential function

$$\mathsf{F}_g(q)=\sum_{\beta}q^\beta\int_{[\overline{Q}_g(X,\beta)]^{vir}}1\,.$$
The integral in the above equation gives nontrivial number since the virtual dimension of $[\overline{Q}_g(X,\beta)]^{vir}$ is $0$ for Calabi-Yau 3-fold.

In \cite{CKg0,CKg}, relationship between $\mathsf{F}_g$ and Gromov-Witten potential function

$$\mathsf{F}^{GW}_g(q)=\sum_{\beta}q^\beta\int_{[\overline{M}_g(X,\beta)]^{vir}}1$$
are studied. Roughly, $\mathsf{F}_g$ are obtained from $\mathsf{F}^{GW}_g$ by change of coordinate $T^X$ associated to $X$ called mirror map. See \cite{CKg} for precise statement. 

In physics, to the A-model Gromov-Witten invariants of a Calabi-Yau 3-fold X is conjecturally associated the B-model theory of a mirror Calabi-Yau 3-fold $\widehat{X}$ (\cite{BCOV,BKMP}). The A-model genus $g$ Gromov-Witten potential function $\mathbf{F}^{A}_g$ which is eqaul to the Gromove-Witten potential function $\mathsf{F}^{GW}_g$ defined above, is  conjecturally related by string theoretic B-model genus $g$ potential function\footnote{$\mathbf{F}^B_g$ is some holomorphic limit of non-holomorphic potential function (\cite{BCOV,BKMP}).} $\mathbf{F}^B_g$ by the same mirror map $T^X$. Therefore the quasimap invariants of $X$ are expected to exactly equal the B-model invariants of the mirror $\widehat{X}$. In this paper we study the string theoretic B-model theory directly via quasimap theory.

\section{Genus zero theory}

We review here the genus $0$ quasimap theory of some possibly non-compact Calabi-Yau manifold $X$. In our paper X will be either total space of canonical bundle or zero loci of general section of anti-canonical bundle over some toric variety $Y$. In either cases, we can study the geometry of $X$ via the Euler class of the obstruction bundle,
\begin{align}\label{Obs}
    e(\text{Obs(X)}),
\end{align}
associated to X on the moduli space $\overline{Q}_{g,n}(Y,\beta)$. Here we also assume GIT representation 
$$\tilde{Y}/\!\!/_{\!\theta}\mathsf{G}$$
of $Y$ with appropriate choice of character $\theta$ and affine variety $\tilde{Y}$. The obstruction bundle in \eqref{Obs} come from some vector bunlde $\tilde{E}$ on $[\tilde{Y}/\mathsf{G}]$ associated to geometry of $X$. Let complex torus $\mathsf{T}$ act on $\tilde{Y}$. We assume that this action commutes with the $\mathsf{G}$ action on $\tilde{Y}$. Assume furthermore, the induced action on $Y$ has only finitely many $0$-dimensional and $1$-dinmensional $\mathsf{T}$-orbits.

\subsection{First correlators.}
We introduce several correlators defined via virtual fundamental class
$$[\overline{Q}_{0,n}(X,\beta)]^{vir}=e(\text{Obs(X)})\cap[\overline{Q}_{0,n}(Y,\beta)]^{vir}$$.
For $\gamma\in H^*_{\mathsf{T}}(X)$, define following series,
\begin{align*}
 &\lan \gamma \psi^{a_1},\dots,\gamma_n \psi^{a_n}\ran_{0,n,\beta} =\int_{[\overline{Q}_{0,n}(X,\beta)]^{vir}}  \prod_{i=1}^{n}\text{ev}^*_i(\gamma_i)\psi^{a_i}\,,\\
 &\lan \gamma \psi^{a_1},\dots,\gamma_n \psi^{a_n}\ran_{0,n}=\sum_{d\ge 0}q^{\beta}\lan \gamma \psi^{a_1},\dots,\gamma_n \psi^{a_n}\ran_{0,n,\beta}\,.
\end{align*}

\subsection{Infinitesimal markings.}\label{lightm} Moduli space of stable quasimaps can be considered with $n$ ordinary (weight 1) markings and $k$ infinitesimal (weight 0+) markings,
$$\overline{Q}^{0+}_{0,n|k}(X,d)\,.$$
Let $\gamma_i\in H^*_{\mathsf{T}}(X)$ be equivariant cohomology classes, and let
$$\delta_j\in H^*_{\mathsf{T}}([\tilde{X}/\mathsf{G}])$$
be classes on the stack quotient. Following the notation of \cite{KL}, define following series,
\begin{multline*}
    \lan \gamma_1\psi^{a_1},\dots,\gamma_n\psi^{a_n};\delta_1,\dots,\delta_k\ran^{0+}_{g,n|k,d}=\\
    \int_{[\overline{Q}^{0+}_{g,n|k}]^{vir}}\prod_{i=1}^n \text{ev}^*_i(\gamma_i)\psi_i^{a_i}\cdot\prod_{j=1}^k\widehat{\text{ev}}^*_j(\delta_j)\,, 
\end{multline*}
\begin{multline*}
    \lan\lan\gamma_1\psi^{a_1},\dots,\gamma_n\psi^{a_n}\ran\ran^{0+}_{0,n}=\sum_{d \ge0}\sum_{k\ge 0}\frac{q^d}{k!}\lan\gamma_1\psi^{a_1},\dots,\gamma\psi^{a_n};t,\dots,t\ran^{0+}_{0,n|k,d}\,,
\end{multline*}
where, $t\in H^*_{\mathsf{T}}([\tilde{X}/\mathsf{G}])$ in the second series.

For each $\mathsf{T}$-fixed point $p_i\in Y$, let $\phi_i$ be the basis of $H^*_{\mathsf{T}}(Y)\ot \CC(\lambda)$ such that

\begin{align*}
 &\phi_i|_{p_j} = \left\{ \begin{array}{rl} 1 & \text{if } i =j   \\
                                            0 & \text{if } i \ne j  \, . \end{array}\right. 
\end{align*}

Let $\phi^i$ be the dual basis with respect to the {\em E-twisted} $\mathsf{T}$-equivariant Poincare pairing, i.e.,
\begin{align*}
 &\int_Y\phi_i \phi^j e^{\mathsf{T}}(\tilde{E}|_Y) = \left\{ \begin{array}{rl} 1 & \text{if } i =j   \\
                                            0 & \text{if } i \ne j  \, , \end{array}\right. 
\end{align*}
where $e^{\mathsf{T}}(\tilde{E}|_Y)$ is the $\mathsf{T}$-equivariant Euler class of $\tilde{E}|_Y$.
Note that 
\begin{align*}
    \phi^i=e_i \phi_i\,\,,\,\,\text{where}\,\,e_i:=\frac{e^{\mathsf{T}}(T_{p_i}Y)}{e^{\mathsf{T}}(\tilde{E}|_{p_i})}
\end{align*}

Following two series play the important role throughout the paper.
\begin{align*}
 &\mathbf{S}_i(\gamma)=e_i\lan\lan\frac{\phi_i}{z-\psi},\gamma \ran\ran^{0+}_{0,2}\, ,\\
 &\mathbf{V}_{ij}=\lan\lan \frac{\phi_i}{x-\psi},\frac{\phi_j}{y-\psi}\ran\ran^{0+}_{0,2} \, .
\end{align*}
We also write

$$\mathbf{S}(\gamma):=\sum_i \phi_i \mathbf{S}_i(\gamma)\,.$$

We recall the following basic result proven in \cite{CKg},

\begin{align*}
 e_i \mathbf{V}_{ij}e_j=\frac{\sum_{k=0}\mathbf{S}_i(\phi_k)|_{z=x}\mathbf{S}_j(\phi^k)|_{z=y}}{x+y}.
\end{align*} 
 
 Associated to each $\mathsf{T}$-fixed point $p_i \in Y$, there is a special $\mathsf{T}$-fixed point locus,
 \begin{align}\label{FL}
  \overline{Q}^{0+}_{0,n|k}(Y,\beta)^{\mathsf{T},p_i}\subset \overline{Q}^{0+}_{0,n|k}(Y,\beta), 
 \end{align}
where all markings lie on a single connected genus $0$ domain component contracted to $p_i$. Denote by Nor the euqivariant normal bundle of $\overline{Q}_{0,m}(Y,\beta)^{\mathsf{T},p_i}$ with respect to the embedding \eqref{FL}. Define following local correlators
\begin{multline*}
 \lan \gamma_1 \psi^{a_i},\dots,\gamma_n \psi^{a_n};\delta_1,\dots,\delta_k\ran^{0+,p_i}_{0,n|k,\beta}=\\
 \int_{[\overline{Q}^{0+}_{0,n|k}(Y,\beta)]^{\mathsf{T},p_i}} \frac{e(\text{Obs})}{e(\text{Nor})}\cdot \prod_{i=1}^n \text{ev}_i^*(\gamma_i)\psi_i^{a_i}\cdot \prod_{j=1}^{k}\widehat{\text{ev}}^*_j(\delta_j)\,,
\end{multline*}
 
\begin{multline*}
    \lann \gamma_1\psi^{a_1},\dots,\gamma_n\psi^{a_n} \rann^{0+,p_i}_{0,n}=\\
    \sum_{d \ge 0}\sum_{\beta \ge 0} \frac{q^\beta}{k!}\lan\gamma_1\psi^{a_1},\dots,\gamma_n\psi^{a_n};t,\dots,t\ran\,.
\end{multline*}

\subsection{Quasimap graph spaces and I-functions.}
\subsubsection{Graph quasimap space.} We review here the constructions of wighted quasimap graph spaces in \cite{BigI}. In this paper, we only consider the case of $(0+,0+)$-stability in \cite{BigI}. See \cite{BigI} for more general constructions. 
 
Following \cite{BigI}, denote by
\begin{align*}
    \mathsf{QG}^{0+,0+}_{g,n|k,d}([\tilde{Y}/\mathsf{G}])\,
\end{align*}
a $(0+,0+)$-stable quasimap graph space. The graph space parametrizes following data
$$((C,\mathbf{x},\mathbf{y}),(f,\rho):C\rightarrow[\tilde{Y}/\mathsf{G}]\times[\CC^2/\CC^*])\,.$$
By the definition of stability condition of the moduli space, $\rho$ is a regular morphism to 
$$\PP^1=\CC^2/\CC^*$$
of degree 1. Therefore, the domain curve $C$ has unique distinguished irreducible component $C_0$ isomorphic to $\PP^1$ via $\rho$. The {\em standard} $\CC^*$-action defined by

\begin{align}\label{act}
t\cdot[\xi_0,\xi_1]=[t\xi_0,\xi_1]\,,\,\,\, \text{for}\,t\in\CC^*\,,\,[\xi_0,\xi_1]\in\PP^1\,,
\end{align}
induces a $\CC^*$-action on $\mathsf{QG}^{0+,0+}_{g,n|k,\beta}([\tilde{Y}/\mathsf{G}])$

Denote by $z$ the weight of $\CC^*$ acting on a point,
$$H^*_{\CC^*}(\text{Spec}(\CC))=\CC[z]\,.$$
Note that $z$ is equal to the $\CC^*$-equivariant first Chern class of the tangent space at $0\in\PP^1$ with respect to the action \eqref{act}, i.e.,
$$z=c_1(T_0\PP^1)\,.$$
 
 The $\mathsf{T}$-action on $Y$ induces a $\mathsf{T}$-action on $\mathsf{QG}^{0+,0+}_{g,n|k,\beta}([Y/\mathsf{G}])$ which commutes with the $\CC^*$-action obtained from the distinguished domain component $C_0$. Therefore, we have a $\mathsf{T}\ti\CC^*$-action on the quasimap graph space and $\mathsf{T}\times\CC^*$-equivariant evaluation maps
 \begin{align*}
     &\text{ev}_i : \mathsf{QG}^{0+}_{g,n|k,\beta}([\tilde{Y}/\mathsf{G}])\rightarrow Y\,, &i=1,\dots,m\,,\\
     &\widehat{\text{ev}}_j : \mathsf{QG}^{0+}_{g,n|k,\beta}([\tilde{Y}/\mathsf{G}])\rightarrow [\tilde{Y}/\mathsf{G}]\,,& j=1,\dots,k\,.
 \end{align*}

\subsubsection{I-functions.} Among the fixed loci for the $\CC^*$-action on 
$$\mathsf{QG}^{0+}_{g,n|k,\beta}([\tilde{Y}/\mathsf{G}])\,,$$
there is a distinguished subset $\mathsf{Q}_{k,\beta}$ for which all the markings and the entire curve class $\beta$ lie over $0\in \PP^1$. The locus $\mathsf{Q}_{k,\beta}$ are eqiuped with a natural {\em proper} evaluation map $ev_{\bullet}$ obtained from the generic point of $\PP^1$:
$$\text{ev}_{\bullet}:\mathsf{Q}_{k,\beta}\rightarrow Y\,.$$

More explicitly,
$$\mathsf{Q} =\mathsf{Q}_\beta \ti 0^k \subset \mathsf{Q}_\beta\ti (\PP^1)^k\,,$$
where $\mathsf{Q}$ is the $\CC^*$-fixed locus in $\mathsf{QG}^{0+}_{0,0,\beta}([\tilde{Y}/\mathsf{G}])$ for which the class $\beta$ is concentrated only over $0\in\PP^1$. The locus $\mathsf{Q}_\beta$ parameterizes quasimaps of class $\beta$,
$$f:\PP^1\rightarrow[\tilde{Y}/\mathsf{G}]\,,$$
with a base-point of length $\beta$ at $0\in\PP^1$. The restriction of $f$ to $\PP^1\setminus \{0\}$ is a constant map to $\PP^2$ which defines the evaluation map $\text{ev}_{\bullet}$

Following \cite{BigI}, we define the big $\mathbf{I}$-function as the generating function for the push-forward via $ev_{\bullet}$ of localization residue contributions of $\mathsf{Q}_{k,\beta}$. For $\mathbf{t}\in H^*_{\mathsf{T}}([\tilde{Y}/\mathsf{G}])\ot_{\CC}\CC[z]$, let
\begin{align*}
    \text{Res}_{\mathsf{Q}_{k,\beta}}(\mathbf{t}^k)&=\prod_{j=1}^k\widehat{\text{ev}}^*_j(\mathbf{t})\cap\text{Res}_{Q_{k,\beta}}[\mathsf{QG}^{0+}_{0,0|k,\beta}([\tilde{Y}/\mathsf{G}])]^{vir}\\
    &=\frac{\prod_{j=1}^k\widehat{\text{ev}}^*_j(\mathbf{t})\cap [\mathsf{Q}_{k,\beta}]^{vir}}{e(\text{Nor}^{vir}_{\mathsf{Q}_{k,\beta}})}
\end{align*}
where $\text{Nor}^{vir}_{\mathsf{Q}_{k,\beta}}$ is the virtual normal bundle.

\begin{Def}\label{BigI} The big $I$-function for the $(0+,0+)$-stability condition is
\begin{align*}
\mathbf{I}(q,\mathbf{t},z)=\sum_{\beta \ge 0}\sum_{k \ge 0}\frac{q^{\beta}}{k!}\text{ev}_{\bullet *}(\text{Res}_{\mathsf{Q}_{k,\beta}}(\mathbf{t}^k))\,.
\end{align*}
\end{Def}

\subsubsection{Birkhoff factorization.} Here we review the procedures in \cite{KL} to obtain $\mathbf{S}$-operators from $\mathbf{I}$-function.

For $\gamma\in H^*_{\mathsf{T}}(Y)$, $\tilde{\gamma}\in H^*_{\mathsf{T}}([\tilde{Y}/\mathsf{G}])$ denotes a lift of $\gamma$, i.e., $\tilde{\gamma}|_Y=\gamma$. We have the following result from \cite[Section 2.3]{KL}.

\begin{Prop} There are unique coefficients $a_i(z,q)\in\CC(\lambda)[z][[q]]$ such that
\begin{align*}
    \sum_i a_i(z,q) z\partial_{\phi_i}\mathbf{I}=\gamma+\mathcal{O}(\frac{1}{z})\,.
\end{align*}
Furthermore LHS coincides with $\mathbf{S}(\gamma)$.
\end{Prop}

\noindent Here we used following notation,

\begin{align*}
    z \partial_{\gamma}\mathbf{I}:=z\frac{\partial}{\partial s}\mathbf{I}(\tilde{t}+s\tilde{\gamma)}|_{s=0}.
\end{align*}

\section{Holomorphic anomaly equation for local $\PP^1\ti\PP^1$} 

\subsection{Overview}Let X be the total space of canonical bundle $K_{\PP^1\ti\PP^1}$ over $\PP^1\ti\PP^1$. Quasimap invariants of $K_{\PP^1\ti\PP^1}$ is defined via the Euler class of the obstruction bundle,
\begin{align}\label{Obs1}
e(\text{Obs}(K_{\PP^1\ti\PP^1}))=e(R^1\pi_*f^*\tilde{E})\,,
\end{align}
associated to the $K_{\PP^1\ti\PP^1}$ geometry on the moduli space $\overline{Q}_{g,n}(\PP^1\ti\PP^1,\beta)$.
Here 
$$f : \mathcal{C}\rightarrow [\CC^2\ti\CC^2/\CC^*\ti\CC^*]\,,\, \pi: \mathcal{C}\rightarrow \overline{Q}_{g,n}(\PP^1\ti\PP^1,(d_1,d_2))$$ 
is the standard universal structures and $\tilde{E}$ is the line bundle on $$[\CC^2\ti\CC^2/\CC^*\ti\CC^*]$$ which induce the canonical bundle $\mathcal{O}_{\PP^1\ti\PP^1}(-2,-2)$ on $\PP^1\ti\PP^1$.

Define genus $g$ quasimap potential function
\begin{align*}
    \mathsf{F}_g(q_1,q_2):=\sum_{d_1,d_2\ge 0}q_1^{d_1}q_2^{d_2}\int_{\overline{Q}_{g,n}(\PP^1\ti\PP^1,(d_1,d_2))}e(\text{Obs}(K_{\PP^1\ti\PP^1})) \,.
\end{align*}
For the proof of holomorphic anomaly equation, we also need following potential function with insertions,
\begin{multline*}
    \mathsf{F}_{g,i}[H_1+H_2,\dots,H_1+H_2](q_1,q_2):=\\
    \sum_{d_1,d_2\ge 0}q_1^{d_1}q_2^{d_2}\int_{\overline{Q}_{g,n}(\PP^1\ti\PP^1,(d_1,d_2))}e(\text{Obs}(K_{\PP^1\ti\PP^1}))\prod_{k=1}^i \text{ev}_k^*(H_1+H_2) \,.
\end{multline*}
Following the notations in Section \ref{No}, define following new series

\begin{align*}
\overline{\mathsf{F}}_g(q):=&\mathsf{F}_g(q,q)\,,\\
\overline{\mathsf{F}}_{g,i}[H_1+H_2,\dots,H_1+H_2](q):=&\mathsf{F}_{g,i}[H_1+H_2,\dots,H_1+H_2](q,q)\,.
\end{align*}

We also have the series $\mathbf{S}_i$, $\mathbf{V}_{ij}$, $\mathbf{I}$ defined in 
Section 2 associated to $K_{\PP^1\ti\PP^1}$ geometry. We will use the notations in Section \ref{No} for these series.
 
 
\subsection{Localization graphs}

\subsubsection{Torus action.}We apply the localization strategy first introduced by Givental for the quasimap invariants of $K_{\PP^1\ti\PP^1}$. Let $\mathsf{T}:=(\CC^*)^4$ act diagonally on  vector space $\CC^2\ti\CC^2$ with weights 
\begin{align*}
    -\lambda_0,-\lambda_1,-\lambda_2,-\lambda_3\,.
\end{align*}
 
 There is an induced $\mathsf{T}$-action on the moduli space $\overline{Q}_{g,n}(\PP^1\ti\PP^1,(d_1,d_2))$. The localization formula of \cite{GP} applied to the virtual fundamenal class $\overline{Q}_{g,n}(\PP^1\ti\PP^1,(d_1,d_2))$ will play a essential role in our paper. The $\mathsf{T}$-fixed loci are represented in terms of dual graphs, and the contributions of the $\mathsf{T}$-fixed loci are given by tautological classes. The formulas here are standard, see \cite{KL,MOP}.
 \subsubsection{Graphs.}\label{locq} Let the genus g and the number of markings $n$ for the moduli space be in the stable range
\begin{align}\label{dmdm}
2g-2+n > 0.
\end{align}
We can organize the $\mathsf{T}$-fixed loci of $\overline{Q}_{g,n}(\PP^1\ti\PP^1,\beta)$ according to decorated graphs. A {\em decorated graph} $\Gamma$(X) consists of the data $(\mathsf{V},\mathsf{E},\mathsf{N},\gamma,\nu)$ where
\begin{enumerate}
 \item [(i)] $\mathsf{V}$ is the vertex set,
 \item [(ii)]$\mathsf{E}$ is the edge set (possibly with self-edges),
 \item [(iii)]$\mathsf{N}:\{1,2,\dots,n\}\rightarrow\mathsf{V}$ is the marking assignment,
 \item [(iv)]$\mathsf{g}:\mathsf{V}\rightarrow\mathbf{Z}_{\ge 0}$ is a genus assignment satisfying
 $$g=\sum_{v\in\mathsf{V}}\mathsf{g}(v)+h^1(\Gamma)$$
 and for which $(\mathsf{V},\mathsf{E},\mathsf{N},\gamma)$ is stable graph,
 \item [(v)]$\mathsf{p}:\mathsf{V}\rightarrow (X)^{\mathsf{T}}$ is an assignment of a $\mathsf{T}$-fixed point $\mathsf{p}(v)$ to each vertex $v\in\mathsf{V}$.
\end{enumerate}
The markings $\mathsf{L}$ are often called {\em legs.}

To each decorated graph $\Gamma\in \mathsf{G}_{g,n}(\PP^1\ti\PP^1)$, we associate the set of fixed loci of 
$$\sum_{d_1,d_2\ge 0}[\overline{Q}_{g,n}(\PP^1\ti\PP^1,(d_1,d_2))]^{vir}q_1^{d_1}q_2^{d_2}$$
whose element can be described as follows:
\begin{enumerate}
\item [(a)] If $\{v_{i_1},\dots,v_{i_k}\}=\{v | \mathsf{p}(v)=p_i\}$, then $f^{-1}(p_i)$ is a disjoint union of connected stable curves of genera $\mathsf{g}(v_{i_1}),\dots,\mathsf{g}(v_{i_k})$.
\item [(b)] There is a bijecive correspondence between the connected components of $f^{-1}(\PP^1\ti\PP^1/\{p_0,\dots,p_m\})$ and the edges of $\Gamma$ respecting vertex incidence. 
Using decorated graph, we can write the localization formula as
$$\sum_{(d_1,d_2) \ge (0,0)}[\overline{Q}_{g,n}(\PP^1\ti\PP^1,(d_1,d_2))]^{vir} q_1^{d_1}q_2^{d_2}=\sum_{\Gamma\in \mathsf{G}_{g,n}(\PP^1\ti\PP^1)}\text{Cont}_{\Gamma}.$$
\end{enumerate}
Note that each contribution $\text{Cont}_{\Gamma}$ is a power series in $q$ obtained from an infinite sum over all edge possibilities (b), while $\mathsf{G}_{g,n}(\PP^1\ti\PP^1)$ is a finite set.
\subsubsection{Unstable graphs}
The moduli spaces of stable quotients $$\overline{Q}_{0,2}(\PP^1\ti\PP^1,(d_1,d_2)) \ \ \
\text{and} \ \ \ \overline{Q}_{1,0}(\PP^1\ti\PP^1,(d_1,d_2))$$
for $d_1>0$ or $d_2>0$
are the only{\footnote{The moduli spaces
$\overline{Q}_{0,0}(\proj^m,d)$ and
$\overline{Q}_{0,1}(\proj^m,d)$
are empty by the definition of a stable
quotient.}}
cases where
the pair $(g,n)$ does 
{\em not} satisfy the Deligne-Mumford stability condition 
\eqref{dmdm}. 

An appropriate set of decorated graphs $\mathsf{G}_{0,2}(\PP^1\ti\PP^1)$
 is easily defined: The graphs
$\Gamma \in \mathsf{G}_{0,2}(\PP^1\ti\PP^1)$ all have 2 vertices
connected by a single edge. Each vertex carries a marking.
All  of the  conditions (i)-(v)
of Section \ref{locq} are satisfied
except for the stability of $(\mathsf{V},\mathsf{E}, \mathsf{N},\gamma)$.
The  localization formula holds,
\begin{eqnarray}\label{ddgg}
\sum_{(d_1,d_2)>(0,0)} \left[\overline{Q}_{0, 2} (\PP^1\ti\PP^1, (d_1,d_2))\right]^{\vir} q_1^{d_1}q_2^{d_2} &=&
\sum_{\Gamma\in \mathsf{G}_{0,2}(\PP^1\ti\PP^1)} \text{Cont}_\Gamma\,,
\end{eqnarray}
For $\overline{Q}_{1,0}(\PP^1\ti\PP^1,(d_1,d_2))$, the matter
is more problematic --- usually a marking
is introduced to break the symmetry.

\subsection{Higher genus series on $K_{\PP^1 \ti \PP^1}$} 
 \subsubsection{Intersection theory on $\overline{M}_{g,n}$.} We review here the now standard method used by Givental \cite{Elliptic,SS,Book} to 
 express genus $g$ descendent correlators in terms of genus 0 data.
 
 Let $t_0,t_1,t_2, \ldots$ be formal variables. The series
 $$T(c)=t_0+t_1 c+t_2 c^2+\ldots$$   in the
additional variable $c$ plays a basic role. The variable $c$
will later be  replaced by the first Chern class $\psi_i$ of
 a cotangent line  over $\overline{M}_{g,n}$, 
 $$T(\psi_i)= t_0 + t_1\psi_i+ t_2\psi_i^2 +\ldots\, ,$$
 with the index $i$
 depending on the position of the series $T$ in the correlator.

Let $2g-2+n>0$.
For $a_i\in \mathbb{Z}_{\geq 0}$ and  $\gamma \in H^*(\overline{M}_{g,n})$, define the correlator 
\begin{multline*}
    \lann \psi^{a_1},\ldots,\psi^{a_n}\, | \, \gamma\,  \rann_{g,n}= 
    \sum_{k\geq 0} \frac{1}{k!}\int_{\overline{M}_{g,n+k}}
    \gamma \, \psi_1^{a_1}\cdots 
     \psi_n^{a_n}  \prod_{i=1}^k T(\psi_{n+i})\, . 
\end{multline*}
For formal variables $x_1,\ldots,x_n$, we also define the correlator
\begin{align}\label{derf}
\lannn \frac{1}{x_1-\psi_1},\ldots,\frac{1}{x_n-\psi_n}\, \Big| \, \gamma \, \rannn_{g,n}
\end{align}
in the standard way by expanding $\frac{1}{x_i-\psi_i}$ as a geometric series.

Denote by $\mathds{L}$ the differential operator 
\begin{align*}
        \mathds{L}\, =\, 
        \frac{\partial}{\partial t_0}-\sum_{i=1}^\infty t_i\frac{\partial}{\partial t_{i-1}}
        \, =\, \frac{\partial}{\partial t_0}-t_1\frac{\partial}{\partial t_0}-t_2\frac{\partial}{\partial t_1}-\ldots
        \, .
\end{align*}
 The string equation yields the following result.
 
\begin{Lemma} \label{stst} For $2g-2+n>0$, we have
$\mathds{L}\lann 1,\ldots,1\, | \, \gamma\, \rann_{g,n}=0$ and 
\begin{multline*}
\mathds{L} \lannn \frac{1}{x_1-\psi_1},\ldots,\frac{1}{x_n-\psi_n}\, \Big| \,\gamma \, 
\rannn_{g,n}= \\
 \left(\frac{1}{x_1}+\ldots +\frac{1}{x_n}\right)
 \lannn\frac{1}{x_1-\psi_1}, \frac{1}{x_n-\psi_n}\, \Big| \, \gamma \, \rannn_{g,n}\, .
 \end{multline*}
\end{Lemma}

After the restriction $t_0=0$ and application of the dilaton equation,
the correlators are expressed in terms of finitely many integrals (by the
dimension constraint). For example,
\begin{eqnarray*}
    \lann 1,1,1\rann_{0,3}\, |_{t_0=0} &= &\frac{1}{1-t_1}\, ,\\
    \lann 1,1,1,1\rann_{0,4}\, |_{t_0=0}& =&\frac{t_2}{(1-t_1)^3}\, ,\\
    \lann 1,1,1,1,1\rann_{0,5}\, |_{t_0=0}&=&\frac{t_3}{(1-t_1)^4}+\frac{3 t_2^2}{(1-t_1)^5}\, ,\\
    \lann 1,1,1,1,1,1\rann_{0,6}\, |_{t_0=0}&=&\frac{t_4}{(1-t_1)^5}+\frac{10 t_2 t_3}{(1-t_1)^6}+\frac{15 t^3_2}{(1-t_1)^7}\, .
\end{eqnarray*}\\

We consider 
$\CC(t_1)[t_2,t_3,...]$
as $\ZZ$-graded ring over $\CC(t_1)$ with 
$$\text{deg}(t_i)=i-1\ \ \text{for $i\geq 2$ .}$$
Define a subspace of homogeneous elements by
$$\CC\left[\frac{1}{1-t_1}\right][t_2,t_3,\ldots]_{\text{Hom}} \subset 
\CC(t_1)[t_2,t_3,...]\, .
$$
We easily see 
$$\lann \psi^{a_1},\ldots,\psi^{a_n}\, | \, \gamma \, \rann_{g,n}\, |_{t_0=0}\ \in\
\CC\left[\frac{1}{1-t_1}\right][t_2,t_3,\ldots]_{\text{Hom}}\, .$$
Using the leading terms (of lowest degree in $\frac{1}{(1-t_1)}$), we obtain the
following result.

\begin{Lemma}\label{basis}
The set of genus 0 correlators
 $$
 \Big\{ \, \lann 1,\ldots,1\rann_{0,n}\, |_{t_0=0} \, \Big\}_{n\geq  4} $$ 
freely generate the ring
 $\CC(t_1)[t_2,t_3,...]$ over $\CC(t_1)$.
\end{Lemma}

By  Lemma \ref{basis}, we can find a unique representation of $\lann \psi^{a_1},\ldots,\psi^{a_n}\rann_{g,n}|_{t_0=0}$
in the  variables
\begin{equation}\label{k3k3}
\Big\{\, \lann 1,\ldots,1\rann_{0,n}|_{t_0=0}\, \Big\}_{n\geq 3}\, .
\end{equation}
The $n=3$ correlator is included in the set \eqref{k3k3} to
capture the variable $t_1$.
For example, in $g=1$,
\begin{eqnarray*}
    \lann 1,1\rann_{1,2}|_{t_0=0}&=&\frac{1}{24}
    \left(\frac{\lann 1,1,1,1,1\rann_{0,5}|_{t_0=0}}{\lan 1,1,1\rann_{0,3}|_{t_0=0}}-\frac{\lann 1,1,1,1\rann^2_{0,4}|_{t_0=0}}{\lann 1,1,1\rann^2_{0,3}|_{t_0=0}}\right)\, ,\\
    \lann 1\rann_{1,1}|_{t_0=0}&=&\frac{1}{24}\frac{\lann 1,1,1,1\rann_{0,4}|_{t_0=0}}{\lann 1,1,1\rann_{0,3}|_{t_0=0}}
    \end{eqnarray*}
A more complicated example in $g=2$ is    
\begin{eqnarray*}
\lann \rann_{2,0}|_{t_0=0}&=& \ \ \frac{1}{1152}\frac{\lann 1,1,1,1,1,1\rann_{0,6}|_{t_0=0}}{\lann 1,1,1\rann_{0,3}|_{t_0=0}^2}\\
    & & -\frac{7}{1920}\frac{\lann 1,1,1,1,1\rann_{0,5}|_{t_0=0}\lann 1,1,1,1\rann_{0,4}|_{t_0=0}}{\lann 1,1,1\rann_{0,3}|_{t_0=0}^3}\\& &+\frac{1}{360}\frac{\lann 1,1,1,1\rann_{0,4}|_{t_0=0}^3}{\lann 1,1,1\rann_{0,3}
    |_{t_0=0}^4}\, .
\end{eqnarray*}

\begin{Def} 
For $\gamma \in H^*(\overline{M}_{g,k})$, let $$\pP^{a_1,\ldots,a_n,\gamma}_{g,n}(s_0,s_1,s_2,...)\in \QQ(s_0, s_1,..)$$ be 
the unique rational function satisfying the condition
$$\lann \psi^{a_1},\ldots,\psi^{a_n}\, |\, \gamma\, \rann_{g,n}|_{t_0=0}
=\pP^{a_1,a_2,...,a_n,\gamma}_{g,n}|_{s_i=\lann 1,\ldots,1\rann_{0,i+3}|_{t_0=0}}\, . $$
\end{Def}
 
\begin{Prop}\label{GR1} For $2g-2+n>0$,
we have
 $$\lann 1,\ldots,1\,|\, \gamma\, \rann_{g,n}
=\pP^{0,\ldots,0,\gamma}_{g,n}|_{s_i=\lann 1,\ldots,1\rann_{0,i+3}}\, . $$
\end{Prop} 

\begin{proof}
 Both sides of the equation satisfy the differential equation
 \begin{align*}
     \mathds{L}=0.
 \end{align*}
 By definition, both sides have the same initial conditions at $t_0=0$.
\end{proof}

\begin{Prop}\label{GR2} For $2g-2+n>0$,
 \begin{multline*}
     \lannn \frac{1}{x_1-\psi_1}, \ldots, \frac{1}{x_n-\psi_n}\, \Big| \, \gamma \, \rannn_{g,n}= \\
     e^{\lann 1,1\rann_{0,2}(\sum_i\frac{1}{x_i})}\sum_{a_1,\ldots,a_n}\frac{\pP^{a_1,\ldots,a_n,\gamma}_{g,n}|_{s_i=\lann 1,\ldots,1\rann_{0,i+3}}
     }{x_1^{a_1+1} \cdots x_n^{a_n+1}}.
 \end{multline*}
\end{Prop} 
 
 \begin{proof}
  Both sides of the equation satisfy differential equation
  \begin{align*}
      \mathds{L}-\sum_i\frac{1}{x_i}=0.
  \end{align*}
 Both sides have the same initial conditions at $t_0=0$.
 We use here
     $$\mathds{L} \lann 1,1\rann_{0,2} =1\,, \ \ \ \  \lann 1,1\rann_{0,2}|_{t_0=0}=0\, .$$
 There is no conflict here with Lemma
 \ref{stst} since $(g,n)=(0,2)$ is not
 in the stable range.
 \end{proof}

\subsubsection{The unstable case $(0,2)$}
The definition given in \eqref{derf}
of the correlator is valid
in the stable range $$2g-2+n>0\, .$$
The unstable case $(g,n)=(0,2)$ plays a
special role. We define
$$\lannn \frac{1}{x_1-\psi_1}, \frac{1}{x_2-\psi_2}\rannn_{0,2}$$
by 
adding the
degenerate term
$$\frac{1}{x_1+x_2}$$
to the terms obtained
by the 
 expansion of $\frac{1}{x_i-\psi_i}$ as 
 a geometric series.
 The degenerate term is associated
to the (unstable) moduli space
of genus 0 with 2 markings.

\begin{Prop}\label{GR22} We have
 \begin{equation*}
     \lannn \frac{1}{x_1-\psi_1}, \frac{1}{x_2-\psi_2} \rannn_{0,2}= 
     e^{\lann 1,1\rann_{0,2}\left(\frac{1}{x_1}+
     \frac{1}{x_2}\right)}\left(\frac{1}{x_1+x_2}\right)\, .
 \end{equation*}
\end{Prop} 
 
 \begin{proof}
  Both sides of the equation satisfy differential equation
  \begin{align*}
      \mathds{L}-\sum_{i=1}^2\frac{1}{x_i}=0.
  \end{align*}
 Both sides have the same initial conditions at $t_0=0$.
 \end{proof}

 \subsection{Decomposition formula.} We review here the localization strategy
to obtain a contribution formula for a general graph $\Gamma$.

\subsubsection{Edge terms}
Recall the definition{\footnote{We use
the variables $x_1$ and $x_2$ here instead
of $x$ and $y$.}}of $\mathbf{V}_{ij}$
given in Section \ref{lightm},
\begin{equation}\label{dfdf6}
\mathbf{V}_{ij}  =  
\Big\langle \Big\langle  \frac{\phi _i}{x- \psi } ,  \frac{\phi _j}{y - \psi } 
\Big\rangle \Big\rangle  _{0, 2}  \, .
\end{equation}
Let $\mathsf{V}_{ij}$ denote
the restriction of $\mathbf{V}_{ij}$
to $t_1=0,t_2=0$.
Via formula \eqref{ddgg},
$\mathsf{V}_{ij}$ is a summation of contributions of fixed loci indexed by
a graph $\Gamma$ consisting of two vertices 
connected by a unique edge. 
Let $w_1$ and $w_2$ be 
$\T$-weights. Denote by $$\mathsf{V}_{ij}^{w_1,w_2}$$ the summation of contributions of $\T$-fixed loci with
tangent weights precisely $w_1$
and $w_2$ on the first rational components
which exit the vertex components over
$p_i$ and $p_j$.

The series $\mathsf{V}_{ij}^{w_1,w_2}$
includes {\em both} vertex and edge
contributions.
By definition \eqref{dfdf6} and the virtual localization formula, we find the
following relationship between
$\mathsf{V}_{ij}^{w_1,w_2}$
and the corresponding
pure edge contribution $\mathsf{E}_{ij}^{w_1,w_2}$,

\begin{eqnarray*}
    e_i\mathsf{V}_{ij}^{w_1,w_2}e_j
    &=& \lannn \frac{1}{w_1-\psi},\frac{1}{x_1-\psi}\rannn^{p_i,0+}_{0,2}\mathsf{E}_{ij}^{w_1,w_2}
    \lannn \frac{1}{w_2-\psi},\frac{1}{x_2-\psi}\rannn^{p_j,0+}_{0,2}\\
    &=&\frac{e^{\frac{\lann 1,1\rann^{p_i,0+}_{0,2}}{w_1}+\frac{\lann 1,1\rann^{p_j,0+}_{0,2}}{x_1}}}{w_1+x_1}
    \, \mathsf{E}^{w_1,w_2}_{ij}\, \frac{e^{\frac{\lann 1,1\rann^{p_i,0+}_{0,2}}{w_2}+\frac{\lann 1,1\rann^{p_j,0+}_{0,2}}{x_2}}}{w_2+x_2}
     \end{eqnarray*}
    
\begin{align*}        
    =\sum_{a_1,a_2}e^{\frac{\lann 1,1\rann^{p_i,0+}_{0,2}}{x_1}+\frac{\lann 1,1\rann^{p_i,0+}_{0,2}}{w_1}}e^{\frac{\lann 1,1\rann^{p_j,0+}_{0,2}}{x_2}+\frac{\lann 1,1\rann^{p_j,0+}_{0,2}}{w_2}}(-1)^{a_1+a_2} \frac{
    \mathsf{E}^{w_1,w_2}_{ij}}{w_1^{a_1}w_2^{a_2}}x_1^{a_1-1}x_2^{a_2-1}\, .
\end{align*}
After summing over all possible weights, we obtain
$$
    e_i\left(\mathsf{V}_{ij}-\frac{\delta_{ij}}{e_i(x+y)}\right)e_j=\sum_{w_1,w_2} e_i\mathsf{V}_{ij}^{w_1,w_2}e_j\, .$$
The above calculations immediately yield
the following result.
    

\begin{Lemma}\label{Edge} We have
 \begin{multline*}
 \left[e^{-\frac{\lann1,1\rann^{p_i,0+}_{0,2}}{x_1}}
       e^{-\frac{\lann1,1\rann^{p_j,0+}_{0,2}}{x_2}}e_i\left(\mathsf{V}_{ij}-\frac{\delta_{ij}}{e_i(x+y)}\right)e_j\right]_{x_1^{a_1-1}x_2^{a_2-1}}=\\
       \sum_{w_1,w_2}
       e^{\frac{\lann1,1\rann^{p_i,0+}_{0,2}}{w_1}}e^{\frac{\lann1,1\rann^{p_j,0+}_{0,2}}{w_2}}(-1)^{a_1+a_2}\frac{\mathsf{E}_{ij}^{w_1,w_2}}{w_1^{a_1}w_2^{a_2}}\, .
 \end{multline*}
\end{Lemma}

\noindent The notation $[\ldots]_{x_1^{a_1-1}x_2^{a_2-1}}$
in Lemma \ref{Edge} denotes the coefficient of
 $x_1^{a_1-1}x_2^{a_2-1}$ in the series expansion
 of the argument.

\subsection{A simple graph}\label{simgr}
Before treating the general case, we present
the localization formula for a simple graph{\footnote{We
follow here the notation of Section \ref{locq}.}.
Let
$\Gamma\in \mathsf{G}_{g}(\PP^1\ti\PP^1)$ 
consist of two vertices   and one edge,
$$v_1,v_2\in \Gamma(V)\, , \ \ \ \ 
e\in \Gamma(E)\, $$
with genus and $\T$-fixed point assignments
$$\mathsf{g}(v_i)=g_i\, , \ \ \ \ \mathsf{p}(v_i)=p_i\, .$$

Let $w_1$ and $w_2$ be tangent
weights at the vertices $p_1$ and $p_2$
respectively. Denote by $\text{Cont}_{\Gamma,w_1,w_2}$
the summation of contributions to
\begin{equation}\label{zlzl}
\sum_{(d_1,d_2)\ge(0,0)} q_1^{d_1}q_2^{d_2}\, \left[\overline{Q}_{g}(K_{\PP^1\ti\PP^1},(d_1,d_2))\right]^{\vir}
\end{equation}
of $\T$-fixed loci with
tangent weights precisely $w_1$
and $w_2$ on the first rational components
which exit the vertex components over
$p_1$ and $p_2$.
We can express the localization formula for 
\eqref{zlzl} as
$$
\lannn \frac{1}{w_1-\psi}\, \Big|\, \mathsf{H}_{g_1}^{p_1}
\rannn_{g_1,1}^{p_1,0+}
\mathsf{E}^{w_1,w_2}_{12} \lannn\frac{1}{w_2-\psi}\, \Big|\, \mathsf{H}_{g_2}^{p_2}
\rannn_{g_2,1}^{p_2,0+} $$
which equals
$$\sum_{a_1,a_2} e^{\frac{\lann1,1\rann^{p_1,0+}_{0,2}}{w_1}}\frac{\ppl
{\psi^{a_1-1}} \, \Big|\, \mathsf{H}_{g_1}^{p_1}     \ppr_{g_1,1}^{p_1,0+}} {w_1^{a_1}} \mathsf{E}^{w_1,w_2}_{12} e^{\frac{\lann 1,1\rann^{p_2,0+}_{0,2}}{w_2}}\frac{\ppl {\psi^{a_2-1}} \, \Big|\, \mathsf{H}_{g_2}^{p_2}\ppr_{g_2,1}^{p_2,0+}}{w_2^{a_2}}
$$
where $\mathsf{H}_{g_i}^{p_i}$ is
the Hodge class 

$$\frac{e(\mathbb{E}^*_g\ot T_{p_i}(\PP^1\ti\PP^1))}{e(T_{p_i}(\PP^1\ti\PP^1))}\cdot\frac{e(\mathbb{E}^*_g\ot(-2\lambda_{p_i}))}{-2(\lambda_{p_i})}\,,$$
with $\lambda_{p_i}:=(H_1+H_2)|_{p_i}$.
We have used here
the notation
\begin{multline*}
\ppl
\psi^{k_1}_1, \ldots,\psi^{k_n}_n \, \Big|\, \mathsf{H}_{h}^{p_i}     \ppr_{h,n}^{p_i,0+} 
=\\
\pP^{k_1,\ldots,k_n,\mathsf{H}_{h}^{p_i}  }_{h,1}\big(\lann 1,1,1\rann_{0,3}^{p_i,0+},\lann 1,1,1,1\rann_{0,4}^{p_i,0+},\ldots \big)
\,
\end{multline*}
and applied \eqref{GR2}.

After summing over all possible weights $w_1,w_2$ and
applying 
Lemma \ref{Edge}, we obtain the following result for the full contribution $$\text{Cont}_\Gamma = \sum_{w_1,w_2} \text{Cont}_{\Gamma,w_1,w_2}$$
of $\Gamma$ to $\sum_{(d_1,d_2)> (0,0)} q_1^{d_1}q_2^{d_2} \left[ \overline{Q}_{g}(K_{\PP^1\ti\PP^1},(d_1,d_2))\right]^{\vir}$.

\begin{Prop} We have \label{propsim}
 \begin{multline*}
     \text{\em Cont}_{\Gamma}=
     \sum_{a_1,a_2>0}
     \ppl
{\psi^{a_1-1}}  \, \Big|\, \mathsf{H}_{g_1}^{p_i}\,     \ppr_{g_1,1}^{p_i,0+}
\ppl
{\psi^{a_2-1}}  \, \Big|\, \mathsf{H}_{g_2}^{p_j}\,     \ppr_{g_2,1}^{p_j,0+}\ \ \ \ \ \ \ \ \ \ \ \\
\ \ \ \ \ \ \ \ \ \ \cdot
     (-1)^{a_1+a_2}\left[e^{-\frac{\lann1,1\rann^{p_i,0+}_{0,2}}{x_1}}
       e^{-\frac{\lann1,1\rann^{p_j,0+}_{0,2}}{x_2}}e_i\left(\mathsf{V}_{ij}-\frac{\delta_{ij}}{e_i(x_1+x_2)}\right)e_j\right]_{x_1^{a_1-1}x_2^{a_2-1}}\, .
 \end{multline*}
\end{Prop}

\subsection{A general graph} We apply the argument of Section \ref{simgr}
to obtain a contribution formula for a general graph $\Gamma$.

Let $\Gamma\in \mathsf{G}_{g,0}(\PP^1\ti\PP^1)$ be a decorated graph as defined in Section \ref{locq}. The {\em flags} of $\Gamma$ are the 
half-edges{\footnote{Flags are either half-edges or markings.}}. Let $\mathsf{F}$ be the set of flags. 
Let
$$\mathsf{w}: \mathsf{F} \rightarrow \text{Hom}(\T, \com^*)\otimes_{\mathbb{Z}}{\mathbb{Q}}$$
be a fixed assignment of $\T$-weights to each flag.

We first consider the contribution $\text{Cont}_{\Gamma,\mathsf{w}}$ to 
$$\sum_{(d_1,d_2)\ge (0,0)} q_1^{d_1}q_2^{d_2} \left[\overline{Q}_g(K_{\PP^1\ti\PP^1},(d_1,d_2))\right]^{\text{vir}}$$
of the $\T$-fixed loci associated $\Gamma$ satisfying
the following property:
the tangent weight on
the first rational component corresponding
to each $f\in \mathsf{F}$ is
exactly given by $\mathsf{w}(f)$.
We have 
\begin{equation}
    \label{s234}
    \text{Cont}_{\Gamma,\mathsf{w}} = \frac{1}{|\text{Aut}(\Gamma)|}
    \sum_{\mathsf{A} \in \ZZ_{> 0}^{\mathsf{F}}} \prod_{v\in \mathsf{V}} \text{Cont}^{\mathsf{A}}_{\Gamma,\mathsf{w}} (v)\prod _{e\in \mathsf{E}} {\text{Cont}}_{\Gamma,\mathsf{w}}(e)\, .
\end{equation}
 The terms on the  right side of \eqref{s234} 
require definition:
\begin{enumerate}
\item[$\bullet$] The sum on the right is over 
the set $\ZZ_{> 0}^{\mathsf{F}}$ of
all maps 
$$\mathsf{A}: \mathsf{F} \rightarrow \ZZ_{> 0}$$
corresponding to the sum over $a_1,a_2$ in
Proposition \ref{propsim}.
\item[$\bullet$]
For $v\in \mathsf{V}$ with 
$n$ incident
flags with $\mathsf{w}$-values $(w_1,\ldots,w_n)$ and
$\mathsf{A}$-values
$(a_1,a_2,...,a_n)$, 
\begin{align*}
    \text{Cont}^{\mathsf{A}}_{\Gamma,{\mathsf{w}}}(v)=
    \frac{\ppl
\psi_1^{a_1-1}, \ldots,
\psi_n^{a_n-1}
\, \Big|\, \mathsf{H}_{\mathsf{g}(v)}^{\mathsf{p}(v)}\,     \ppr_{\mathsf{g}(v),n}^{\mathsf{p}(v),0+}}
{w_1^{a_1} \cdots w_n^{a_n}}\, .
\end{align*}
\item[$\bullet$]
For $e\in \mathsf{E}$ with 
assignments $(\mathsf{p}(v_1), \mathsf{p}(v_2))$
for the two associated vertices{\footnote{In case $e$
is self-edge, $v_1=v_2$.}} and 
$\mathsf{w}$-values $(w_1,w_2)$ for the two associated flags,
    $$    
    \text{Cont}_{\Gamma,\mathsf{w}}(e)=
    e^{\frac{\lann1,1\rann^{\mathsf{p}(v_1),0+}_{0,2}}{w_1}}
    e^{\frac{\lann1,1\rann^{\mathsf{p}(v_2),0+}_{0,2}}{w_2}}
    \mathsf{E}^{w_1,w_2}_{\mathsf{p}(v_1),\mathsf{p}(v_2)}\, .$$
\end{enumerate}
The localization formula then yields \eqref{s234}
just as in the simple case of Section \ref{simgr}.

By summing the contribution \eqref{s234} of $\Gamma$ over
all the weight functions $\mathsf{w}$
and applying Lemma \ref{Edge}, we obtain
the following result which generalizes 
Proposition \ref{propsim}.

\begin{Prop}\label{VE} We have
 $$
 \text{\em Cont}_\Gamma
     =\frac{1}{|\text{\em Aut}(\Gamma)|}
     \sum_{\mathsf{A} \in \ZZ_{> 0}^{\mathsf{F}}} \prod_{v\in \mathsf{V}} 
     \text{\em Cont}^{\mathsf{A}}_\Gamma (v)
     \prod_{e\in \mathsf{E}} \text{\em Cont}^{\mathsf{A}}_\Gamma(e)\, ,
 $$
 where the vertex and edge contributions 
 with incident flag $\mathsf{A}$-values $(a_1,\ldots,a_n)$
 and $(b_1,b_2)$ respectively are
 \begin{eqnarray*}
    \text{\em Cont}^{\mathsf{A}}_\Gamma (v)&=&
    \ppl
\psi_1^{a_1-1}, \ldots,
\psi_n^{a_n-1}
\, \Big|\, \mathsf{H}_{\mathsf{g}(v)}^{\mathsf{p}(v)}\,
  \ppr_{\mathsf{g}(v),n}^{\mathsf{p}(v),0+}\,  ,\\
    \text{\em Cont}^{\mathsf{A}}_\Gamma(e)
    &=&
    (-1)^{b_1+b_2}\left[e^{-\frac{\lann1,1\rann^{\mathsf{p}(v_1),0+}_{0,2}}{x_1}}
       e^{-\frac{\lann1,1\rann^{\mathsf{p}(v_2),0+}_{0,2}}{x_2}}e_i\left(\mathsf{V}_{ij}-\frac{\delta_{ij}}{e_i(x+y)}\right)e_j\right]_{x_1^{b_1-1}x_2^{b_2-1}}\, ,
 \end{eqnarray*}
where $\mathsf{p}(v_1)=p_i$ and $\mathsf{p}(v_2)=p_j$ in the second equation. 
\end{Prop}
 
\subsubsection{Legs} 
Let $\Gamma \in \mathsf{G}_{g,n}(\PP^1\ti\PP^1)$ be a decorated graph
with markings. While no markings are needed to define the
stable quotient invariants of $K_{\PP^1\ti\PP^1}$, the contributions
of decorated graphs with markings will appear in the
proof of the holomorphic anomaly equation.
The formula for the contribution $\text{Cont}_\Gamma((H_1+H_2)^{k_1},\ldots,(H_1+H_2)^{k_n})$
of $\Gamma$ to 
\begin{align*}
    \sum_{(d_1,d_2)\ge (0,0)}q_1^{d_1}q_2^{d_2} \prod_{j=0}^n \text{ev}^*((H_1+H_2)^{k_j})\cap\left[ \overline{Q}_{g,n}(K_{\PP^1\ti\PP^1},(d_1,d_2))\right]^{\vir} 
\end{align*}
is given by the following result.
\begin{Prop}\label{VEL} We have
 \begin{multline*}
 \text{\em Cont}_\Gamma((H_1+H_2)^{k_1},\ldots,(H_1+H_2)^{k_n})
     =\\\frac{1}{|\text{\em Aut}(\Gamma)|}
     \sum_{\mathsf{A} \in \ZZ_{>0}^{\mathsf{F}}} \prod_{v\in \mathsf{V}} 
     \text{\em Cont}^{\mathsf{A}}_\Gamma (v)
     \prod_{e\in \mathsf{E}} \text{\em Cont}^{\mathsf{A}}_\Gamma(e)
     \prod_{l\in \mathsf{L}} \text{\em Cont}^{\mathsf{A}}_\Gamma(l)\, ,
 \end{multline*}
 where the leg contribution 
 is 
 \begin{eqnarray*}
     \text{\em Cont}^{\mathsf{A}}_\Gamma(l)
    &=&
    (-1)^{\mathsf{A}(l)-1}\left[e^{-\frac{\lann1,1\rann^{\mathsf{p}(l),0+}_{0,2}}{z}}
       \mathsf{S}_{\mathsf{p}(l)}((H_1+H_2)^{k_l})\right]_{z^{\mathsf{A}(l)-1}}\, .
 \end{eqnarray*}
The vertex and edge contributions are same as before.
\end{Prop}

The proof of Proposition \ref{VEL} 
follows the vertex and edge analysis. We leave the
details as an exercise for the reader.
The parallel statement for Gromov-Witten theory
can be found in \cite{Elliptic, SS,Book}.

\subsection{twisted theory on $\PP^3$}\label{TT}
Here we study the genus $0$  twisted quasimap theory of $\PP^3$ by $\mathcal{O}_{\PP^3}(2)\oplus\mathcal{O}_{\PP^3}(-2)$. 
We fix a torus action $\mathsf{T}:=(\CC^*)^4$ on $\PP^3$ with weights
$$-\xi_0, -\xi_1, -\xi_2, -\xi_3$$ on the vector space $\CC^4$. We will use the specialization
\begin{align}\label{Sp}
 \xi_i=\zeta^i\,,\,\,\, \text{for}\,\,i=0,1,2,3,4\,,    
\end{align}
where $\zeta$ is the primitive fourth root of unity.  
We denote by
\begin{align*}
    \mathbf{S}^{\text{tw}}\,,\,\mathbf{V}^{\text{tw}}\,,\,\mathbf{I}^{\text{tw}}\,,\,\mathbf{U}^{\text{tw}}
\end{align*}
 the series defined in Section 2 via the Euler class of the obstruction bundle,
 \begin{align}\label{Obs2}
 e(\text{Obs}(\PP^3,\text{tw}))=e(R^0\pi_*f^*\tilde{\mathcal{O}}_{\PP^3}(2)\oplus R^1\pi_*f^*\tilde{\mathcal{O}}_{\PP^3}(-2))\end{align}
on the moduli space $\overline{Q}_{g,n}(\PP^3,d)$. Here $\tilde{\mathcal{O}}_{\PP^3}(\pm 2)$ is the line bundle on $[\CC^4/\CC^*]$ which induces the line bundle $\mathcal{O}_{\PP^3}(\pm 2)$ on $\PP^3$.
 
\subsubsection{$\mathbf{I}$-fucntion and Picard-Fuchs equation.} Let $\tilde{H}\in H^*_{\mathsf{T}}([\CC^4/\CC^*])$ and $H\in H^*_{\mathsf{T}}(\PP^3)$ denote the respective hyperplane class. The $\mathbf{I}$-function for twisted theory can be evaluated as follows.

\begin{Prop}
 For $\mathbf{t}=t\tilde{H}\in H^*_{\mathsf{T}}([\CC^4/\CC^*],\QQ)$,
 \begin{align}
     \mathbf{I}^{\text{tw}}(t)=\sum_{d=0}^{\infty} q^d e^{t(H+dz)/z}\frac{\prod_{k=0}^{2d-1}(-2H-kz)\prod_{k=0}^{2d}(2H+kz)}{\prod_{i=0}^3\prod_{k=1}^d(H-\xi_i+kz)}\,.
 \end{align}
\end{Prop}

The function $\mathbf{I}^{\text{tw}}$ satisfies following Picard-Fuchs equation

\begin{multline}\label{PF}
    \left((z\frac{d}{dt})^4-1-q(2(z\frac{d}{dt})+z)(2(z\frac{d}{dt})+2z)\right.\\
    \left.\cdot(-2(z\frac{d}{dt}))(-2(z\frac{d}{dt})-z)\right)\mathbf{I}^{tw}=0\,.
\end{multline}

We now recall the functions $\mathbf{S}^{\text{tw}}_i(\gamma)$ defined in Section \ref{lightm}. Using Birkhoff factorization, an evaluation of the following series can be obtained from the $\mathbf{I}$-function, see \cite{KL}.

\begin{align}\label{S}
\nonumber\mathbf{S}^{\text{tw}}(1)=\frac{\mathbf{I}^{\text{tw}}}{\mathbf{I}^{\text{tw}}|_{t=0,H=1,z=\infty}}\,,\\
\mathbf{S}^{\text{tw}}(H)=\frac{z\frac{d}{dt}\mathbf{S}^{\text{tw}}(1)}{z\frac{d}{dt}\mathbf{S}^{\text{tw}}(1)|_{t=0,H=1,z=\infty}}\,,\\
\nonumber\mathbf{S}^{\text{tw}}(H^2)=\frac{z\frac{d}{dt}\mathbf{S}^{\text{tw}}(H)}{z\frac{d}{dt}\mathbf{S}^{\text{tw}}(H)|_{t=0,H=1,z=\infty}}\,,\\
\nonumber\mathbf{S}^{\text{tw}}(H^3)=\frac{z\frac{d}{dt}\mathbf{S}^{\text{tw}}(H^2)}{z\frac{d}{dt}\mathbf{S}^{\text{tw}}(H^2)|_{t=0,H=1,z=\infty}}\,.
\end{align}

For a series $\mathbf{Z}\in \CC[[\frac{1}{z}]]$, the specialization $\mathbf{Z}|_{z=\infty}$ denote constant term of $\mathbf{Z}$ with respect to $\frac{1}{z}$.

\subsubsection{Asymptotic expansion.}
The restriction $\mathbf{I}^{\text{tw}}|_{H=\xi_i}$ admits following asymptotic form

\begin{align*}
    \mathsf{I}^{\text{tw}}|_{H=\xi_i}=e^{\mu\xi_i/z}\left( R_0+R_1(\frac{z}{\xi_i})+R_2(\frac{z}{\xi_i})^2+\dots\right)\,
\end{align*}
for some series $\mu(q),R_k(q)\in\CC[[q]]$. Here $$\mathsf{I}^{\text{tw}}:=\mathbf{I}^{\text{tw}}|_{t=0}\,.$$
The series $\mu$ and $R_k$ are found by solving differential equations obtained from the coefficient of $z^k$ in \eqref{PF}. For example,
\begin{align*}
    1+\mathsf{D}\mu&=L\,,\\
    R_0&=L^{\frac{1}{2}}\,,\\
    R_1&=L^{\frac{1}{2}}(\frac{3}{32L}+\frac{1}{24}-\frac{13}{96}L^3)\,,
\end{align*}
where $L(q)=(1-16q)^{-\frac{1}{4}}$. Here $\mathsf{D}:=q\frac{q}{dq}$.

Now we return to the series $\mathbf{S}^{\text{tw}}(\gamma)$. Define the series $C_k$ for $k=1,2,3$, by
\begin{align}\label{Ck}
   \nonumber C_1&=z\frac{d}{dt}\mathbf{S}^{\text{tw}}(H)|_{t=0,H=1,z=\infty}\,,\\
    C_2&=z\frac{d}{dt}\mathbf{S}^{\text{tw}}(H^2)|_{t=0,H=1,z=\infty}\,,\\
    \nonumber C_3&=z\frac{d}{dt}\mathbf{S}^{\text{tw}}(H^3)|_{t=0,H=1,z=\infty}\,.
\end{align}
 
 As in \cite{ZaZi}, we can prove following relations,
\begin{align*}
C_1&=C_3 \,,\\
C_1C_2C_3&=L^4\,.
\end{align*}

Define new series for $\gamma \in H^*_{\mathsf{T}}(\PP^3,\QQ)$

$$\mathsf{S}^{\text{tw}}_i(\gamma):=\mathbf{S}^{\text{tw}}_i(\gamma)|_{t=0}\,.$$
The function $\mathsf{S}^{\text{tw}}_i(H^k)$ also admit the following asymptotic expansion:

\begin{align}\label{A}
    \nonumber  \mathsf{S}^{\text{tw}}_i(1)&=e^{\frac{\mu \xi_i}{z}}\left(R_{00}+R_{01}(\frac{z}{\xi_i})+R_{02}(\frac{z}{\xi_i})^2+\dots\right)\,,\\
    \mathsf{S}^{\text{tw}}_i(H)&=e^{\frac{\mu \xi_i}{z}}\frac{L\xi_i}{C_1}\left(R_{10}+R_{11}(\frac{z}{\xi_i})+R_{12}(\frac{z}{\xi_i})^2+\dots\right)\,,\\
    \nonumber \mathsf{S}^{\text{tw}}_i(H^2)&=e^{\frac{\mu \xi_i}{z}}\frac{L^2\xi_i^2}{C_1C_2}\left(R_{20}+R_{21}(\frac{z}{\xi_i})+R_{22}(\frac{z}{\xi_i})^2+\dots\right)\,,\\
    \nonumber  \mathsf{S}^{\text{tw}}_i(H^3)&=e^{\frac{\mu \xi_i}{z}}\frac{L^3\xi_i^3}{C_1C_2C_3}\left(R_{30}+R_{31}(\frac{z}{\xi_i})+R_{32}(\frac{z}{\xi_i})^2+\dots\right)\,,
\end{align}

We used here the normalization in \cite{ZaZi}. Note
$$R_{0k}=R_k\,.$$

As in \cite{ZaZi}, we obtain the following results.

\begin{Lemma}\label{R0}
 For all $k\ge 0$, we have
 \begin{align*}
     R_k\in\QQ[L^{\pm 1}]\,.
 \end{align*}
\end{Lemma}

By applying \eqref{S} to the asymptotic form \eqref{A}, we obtain the following results.

\begin{Lemma}\label{R}We have
 \begin{align*}
     R_{1p+1}&=R_{0p+1}+\frac{DR_{0p}}{L}\,,\\
     R_{2p+1}&=R_{1p+1}+\frac{DR_{1p}}{L}+(\frac{DL}{L^2}-\frac{\mathcal{X}}{L})R_{1p}\,,\\
     R_{3p+1}&=R_{2p+1}+\frac{DR_{2p}}{L}+(\frac{\mathcal{X}}{L}-2\frac{DL}{L^2})R_{2p}\,.\\
     R_{0p+1}&=R_{3p+1}+\frac{DR_{3p}}{L}-\frac{DL}{L^2}R_{3p}\,.
 \end{align*}
 where $\mathcal{X}:=\frac{DC_1}{C_1}$.
\end{Lemma}

The last equality come from 
\begin{align*}
    \mathbf{S}^{\text{tw}}(1)=\mathbf{S}^{\text{tw}}(H^4)=\frac{z\frac{d}{dt}\mathbf{S}^{\text{tw}}(H^3)}{z\frac{d}{dt}\mathbf{S}^{\text{tw}}(H^3)|_{t=0,H=1,z=\infty}}\,.
\end{align*}

After setting $p=1$ in Lemma \ref{R}, we obtain the following relation:

\begin{align}\label{drule}
    \mathcal{X}^2-(L^4-1)\mathcal{X}-\frac{1}{4}(L^4-1)+\mathsf{D}\mathcal{X}=0\,.
\end{align}

From Proposition \ref{R0}, Lemma \ref{R} and \eqref{drule}, we obtain the following results.

\begin{Lemma}\label{Rp}
  For all $k\ge 0$, we have
  \begin{align*}
      R_{1k}\,, \,R_{3k}\in\CC[L^{\pm 1}]\,,\\
      R_{2k}=Q_{2k}-\frac{R_{1k-1}}{L}\mathcal{X}\,,
  \end{align*}
  for some $Q_{2k}\in\CC[L^{\pm 1}]$.
\end{Lemma}

\subsection{Quantum Lefschetz hyperplane section theorem.}

To apply the higher genus formula for $K_{\PP^1\ti\PP^1}$, one need to know the series $\mathbf{S}(\gamma)$ for $\gamma\in H_{\mathsf{T}}^*(K_{\PP^1\ti\PP^1})$. Since it is dificult to study $\mathbf{S}$-operator of $K_{\PP^1\ti\PP^1}$ directly, we study it via $\mathbf{S}$-operator of twisted theory of $\PP^3$ studied in Section \ref{TT}. 

 Consider the Segre embedding 
\begin{align*}
    \iota:\PP^1\ti\PP^1\rightarrow\PP^3\,.
\end{align*}

Let $\mathsf{T}=(\CC^*)^2\ti(\CC^*)^2$ act on $\PP^1\ti\PP^1$ diagonaly on each component with weight $\lambda_0,\lambda_1,\lambda_2,\lambda_3$. This action uniquely lift to $\PP^3$. We will use the following specializations:
\begin{align*}
    \lambda_0=\frac{1+i}{2}\,,\lambda_1=-\frac{1+i}{2}\,,
    \lambda_2=\frac{1-i}{2}\,,
    \lambda_3=-\frac{1-i}{2}\,.
\end{align*}
Note above specialization is consistent with the specialization \eqref{Sp}.

Denote by $H^*_{\mathsf{T}}(K_{\PP^1\ti\PP^1})$ the cohomology ring of $K_{\PP^1\ti\PP^1}$ whose ring sturucture is equal to usual cohomology ring $H^*_{\mathsf{T}}(\PP^1\ti\PP^1)$ of $\PP^1\ti\PP^1$ with following cup product,

\begin{align*}
    a\cup_{K_{\PP^1\ti\PP^1}}b&=\int_{\PP^1\ti\PP^1}a\cup b\cup(-\frac{1}{2(H_1+H_2)})\,.
\end{align*}
Denote by $H^*_{\mathsf{T},tw}(\PP^3)$ the twisted cohomology ring of $\PP^3$ whose ring strucutre is equal to usual cohomology ring $H^*_{\mathsf{T}}$ of $\PP^3$ with following cup product,

\begin{align*}
     a\cup_{\PP^3,tw} b&=\int_{\PP^3}(-1)a\cup b\,.
\end{align*}
Pull-back map induced by $\iota$ gives us following isomorphism between above two cohomology rings with associated cup-products.

\begin{align}\label{Iso}
    \iota^*:H^*_{\mathsf{T},tw}(\PP^3)\rightarrow H^*_{\mathsf{T}}(K_{\PP^1\ti\PP^1})\,.
\end{align}

Let us use the same notation $\iota$ to denote the closed immersion of $\overline{Q}_{g,k}(\PP^1\ti\PP^1,(d_1,d_2))$ into $\overline{Q}_{g,k}(\PP^3,d)$. Recall the definition of obstruction bundle associated to each geometry \eqref{Obs1},\eqref{Obs2}. By the functoriality in \cite{KKP} we obtain the following results.

\begin{Prop}\label{QHT}
 We have
 \begin{multline*}
     \sum_{d_1+d_2=d}\iota_*(e(\text{{\em Obs}}(K_{\PP^1\ti\PP^1}))\cap[\overline{Q}_{g,k}(\PP^1\ti\PP^1,(d_1,d_2))])=\\e(\text{{\em Obs}}(\PP^3,\text{\em tw}))\cap [\overline{Q}_{g,k}(\PP^3,d)]\,.
 \end{multline*}
\end{Prop}


The following is immediate consequences of Proposition \ref{QHT}.

\begin{Cor} \label{SS}Under the isomorphism of \eqref{Iso}, we have 
\begin{align*}
    \overline{\mathsf{S}}((H_1+H_2)^k)=\iota^*(\mathsf{S}^\text{\em tw}(H^k))\,\,\,\text{for}\,\,k=0,1,2,3.
\end{align*}
 
\end{Cor}

\subsection{Proof of Theorem \ref{Th1}}

By definition, we have
\begin{equation}\label{ffww}
A_2(q)= \frac{1}{L^4}\left(\mathcal{X}
+\frac{1}{2} -\frac{L^4}{4}\right)\, .
\end{equation}
Hence, statement (i),
$$\overline{\mathsf{F}}_g (q) \in \mathbb{C}[L^{\pm1}][A_2]\, ,$$
follows from Proposition \ref{VE}
and  Lemmas \ref{Rp}.
Since 
$$\frac{\partial}{\partial T} = \frac{1}{C_1}q\frac{ \partial}{\partial q}\,, $$
statement (iii),
\begin{equation}\label{vvtt}
\frac{\partial^k \overline{\mathsf{F}}_g}{\partial T^k}(q) \in \mathbb{C}[L^{\pm1}][A_2][C_1^{-1}]\, ,
\end{equation}
follows since the ring
$$\mathbb{C}[L^{\pm1}][A_2]=\mathbb{C}[L^{\pm1}][X]$$
is closed under the action of the differential operator $$\DD=q\frac{\partial}{\partial q}\, $$
by \eqref{drule}.
\qed

\subsection{Proof of Theorem \ref{Th2}}

Let $\Gamma\in\mathsf{G}_g(\PP^1\ti\PP^1)$ be a decorated graph. Let us fix an edge $f\in \mathsf{E}(\Gamma)$. If we break the $\Gamma$ at the edge $f$, we have two possibilities:
\begin{itemize}
 \item if $\Gamma$ is connected after breaking $f$, denote the resulting graph by
 $$\Gamma^0_f\in\mathsf{G}_{g-1,2}(\PP^1\ti\PP^1)$$
 \item if $\Gamma$ is disconnected after breaking $f$, denote the resulting two graphs by
 $$\Gamma^1_f\in\mathsf{G}_{g_1,1}(\PP^1\ti\PP^1)\,\, \text{and}\,\,\Gamma^2_f\in\mathbf{G}_{g_2,1}(\PP^1\ti\PP^1)$$
 where $g=g_1+g_2$.
\end{itemize}
There is no canonical order for the 2 new markings. We will always sum over the 2 labellings. So more precisely, the graph $\Gamma^0_f$ in case $\bullet$ should be viewed as sum of 2 graphs
$$\Gamma^0_{f,(1,2)}+\Gamma^0_{f,(2,1)}\,.$$
Similarly, in case $\bullet\bullet$, we will sum over the ordering of $g_1$ and $g_2$. As usual, the summation will be later compensated by a factor of $\frac{1}{2}$ in the formulas.

By Proposition \ref{VE}, we have
the following formula for the contribution 
of the graph $\Gamma$ to the quasimap
theory of $K_{\PP^1\ti\PP^1}$,
 $$
 \overline{\text{Cont}}_\Gamma
     =\frac{1}{|\text{Aut}(\Gamma)|}
     \sum_{\mathsf{A} \in \ZZ_{\ge 0}^{\mathsf{F}}} \prod_{v\in \mathsf{V}} 
     \overline{\text{Cont}}^{\mathsf{A}}_\Gamma (v)
     \prod_{e\in \mathsf{E}} \overline{\text{Cont}}^{\mathsf{A}}_\Gamma(e)\, .
 $$
The bar over $\text{Cont}_{\Gamma}$ in the above equation means restriction $q_1=q_2=q$ following the notation in Section \ref{No}.

Let $f$ connect the $\T$-fixed points $p_i, p_j \in \PP^1\ti\PP^1$. Let
the $\mathsf{A}$-values of the respective
half-edges be $(k,l)$. By Lemma \ref{Rp} and Corollary \ref{SS}, we have
\begin{equation}\label{Coeff}
\frac{\partial \overline{\text{Cont}}^{\mathsf{A}}_\Gamma(f)}{\partial \mathcal{X}} = (-1)^{k+l}\frac{ R_{1k-1} R_{1l-1}}{L\lambda_i^{k-2}\lambda_j^{l-2}}\, .
\end{equation}

\noindent $\bullet$ If $\Gamma$ is connected after breaking $f$, we have
\begin{multline*}
\frac{1}{|\text{Aut}(\Gamma)|}
     \sum_{\mathsf{A} \in \ZZ_{\ge 0}^{\mathsf{F}}} 
     \left(\frac{L^4}{C^2_1}\right)
     \frac{\partial {\overline{\text{Cont}}}^{\mathsf{A}}_\Gamma(f)}{\partial \mathcal{X}} 
     \prod_{v\in \mathsf{V}} 
     \overline{\text{Cont}}^{\mathsf{A}}_\Gamma (v)
     \prod_{e\in \mathsf{E},\, e\neq f} \overline{\text{Cont}}^{\mathsf{A}}_\Gamma(e) \\=\frac{1}{2}
\overline{\text{Cont}}_{\Gamma^0_f}(H_1+H_2,H_1+H_2) \, .
\end{multline*}
The derivation is simply by using \eqref{Coeff} on the left
and Proposition \ref{VEL} on the right.

\vspace{5pt}
\noindent $\bullet\bullet$
If $\Gamma$ is disconnected after breaking $f$, we obtain
\begin{multline*}
\frac{1}{|\text{Aut}(\Gamma)|}
     \sum_{\mathsf{A} \in \ZZ_{\ge 0}^{\mathsf{F}}} 
     \left(\frac{L^4}{C^2_1}\right)
     \frac{\partial {\overline{\text{Cont}}}^{\mathsf{A}}_\Gamma(f)}{\partial \mathcal{X}} 
     \prod_{v\in \mathsf{V}} 
     \overline{\text{Cont}}^{\mathsf{A}}_\Gamma (v)
     \prod_{e\in \mathsf{E},\, e\neq f} \overline{\text{Cont}}^{\mathsf{A}}_\Gamma(e)\\
=\frac{1}{2}\overline{\text{Cont}}_{\Gamma^1_f}(H_1+H_2) \,
\overline{\text{Cont}}_{\Gamma^2_f}(H_1+H_2)\, 
\end{multline*}
by the same method.

By combining the above two equations for all 
the edges of all the graphs $\Gamma\in \mathsf{G}_g(\PP^1\ti\PP^1)$
and using the vanishing
\begin{align*}
\frac{\partial {\overline{\text{Cont}}}^{\mathsf{A}}_\Gamma(v)}{\partial \mathcal{X}}=0
\end{align*}
of Lemma \ref{R0}, we obtain
\begin{multline}\label{greww}
\left(\frac{L^4} {C^2_1}\right) \frac{\partial}{\partial \mathcal{X}} 
\overline{\mathsf{F}}_g=\\ \frac{1}{2}\sum_{i=1}^{g-1}\overline{F}_{g-i,1}[H_1+H_2]
\overline{\mathsf{F}}_{i,1}[H_1+H_2] + \frac{1}{2} \overline{\mathsf{F}}_{g-1,2}[H_1+H_2,H_1+H_2]\, .
\end{multline}

Combining the divisor equation in Gromov-Witten theory and the wall-crossing theorem, we obtain the divisor equation in quasimap invariants.
\begin{align*}
    \overline{\mathsf{F}}_{g,1}[H_1+H_2]=\frac{\partial \overline{\mathsf{F}}_g}{\partial T}\,,\\
    \overline{\mathsf{F}}_{g,2}[H_1+H_2,H_1+H_2]=\frac{\partial^2 \overline{\mathsf{F}}_g}{\partial T^2}
\end{align*}

The above equations transform \eqref{greww} into exactly the holomorphic anomaly equation in Theorem \ref{Th2},

$$\frac{1}{C_1^2}\frac{\partial \overline{\mathsf{F}}_{g}}{\partial{A_2}}(q)
= \frac{1}{2}\sum_{i=1}^{g-1} 
\frac{\partial \overline{\mathsf{F}}_{g-i}}{\partial{T}}(q)
\frac{\partial \overline{\mathsf{F}}_{i}}{\partial{T}}(q)
+
\frac{1}{2}
\frac{\partial^2 \overline{\mathsf{F}}_{g-1}}{\partial{T}^2}(q)\,
$$
as an equality in $\CC[[q]]$. We can lift the above identity to the ring $\CC[L^{\pm 1}][A_2,C_1^{-1}]$ by the same method in \cite[Section 7.2]{LP1}.
\qed

\

\section{Genus one quasimap invariant of Calabi-Yau hypersurface in toric variety.} 
\subsection{Overview}
We review here the method in \cite{KL} of computing genus one quasimap invariants.


Let X be a generic section of degree $(m,n)$ in $Y:=\PP^{m-1}\ti\PP^{n-1}$. The quasimap invariants of $X$ is defined in Section 2. Now consider the imbedding

$$\iota:X \rightarrow Y\,.$$

By functoriality in \cite{KKP} we have following identity for $g=0,1$,

\begin{align}\label{Func}
\iota_*([\overline{Q}_{g,0}(X,(d_1,d_2))]^{\text{vir}})=e(\text{Obs})\cap [\overline{Q}_{g,0}(Y,(d_1,d_2))]^{\text{vir}}\,,
\end{align}
where
$$
e(\text{Obs})=e(R^1\pi_*f^*\tilde{E})\,.$$
Here $\tilde{E}$ is the line bundle on
$$[\CC^m\ti\CC^n / \CC^*\ti\CC^*]$$ which induces the line bundle $\mathcal{O}_{\PP^{n-1}\ti\PP^{m-1}}(n,m)$ on $\PP^{n-1}\ti\PP^{m-1}$.

$$f : \mathcal{C}\rightarrow [\CC^m\ti\CC^n/\CC^*\ti\CC^*]\,,\, \pi: \mathcal{C}\rightarrow \overline{Q}_{g,0}(\PP^{m-1}\ti\PP^{n-1},(d_1,d_2))$$
is the standard universal structures. While there is no torus action on $X$, we can use torus localization theorem on the right side of \eqref{Func} with $\mathsf{T}:=(\CC^*)^m\ti(\CC^*)^n$ acting on $\PP^{m-1}\ti\PP^{n-1}$ standardly. 

The genus one quasimap potential function of $X$ was defined in Section 1,

\begin{align*}\mathsf{F}_1(q_1,q_2):=&\sum_{d_1,d_2\ge 0}q_1^{d_1}q_2^{d_2}\int_{[\overline{Q}_{1,0}(X,(d_1,d_2))]^{\text{vir}}}1\,\\
=&\sum_{d_1,d_2\ge 0}q_1^{d_1}q_2^{d_2}\int_{[\overline{Q}_{1,0}(\PP^{m-1}\ti\PP^{m-1},(d_1,d_2))]^{\text{vir}}}e(\text{Obs}) \,.
\end{align*}

Now we recall the series defined in Section 2 for X. From \cite{CKg0}, we have following asymtotic expansion of the function $\mathbf{I}$ after restriction to each fixed point $p_i\in (Y)^{\mathsf{T}}$,

\begin{align*}
    \mathbf{I}|_{p_i}=e^{\frac{\mathbf{U}_i}{z}}\left(\sum_{k=0}^{\infty}\mathbf{R}_{ik}z^k\right)
\end{align*}
for some unique series $\mathbf{R}_{ik}\in\CC(\lambda)[[t_1,t_2,q_1,q_2]]$.
Here $$\mathbf{U}_i=\lan\lan1,1\ran\ran^{0+,p_i}_{0,2}\,.$$

\noindent Let $c_i(\lambda)$ be the element in $\CC(\lambda)$ uniquely determined by 
$$1+c_i(\lambda)e(\mathbb{E})=\frac{e^{\mathsf{T}}(\mathbb{E}^{*}\otimes T_{p_i}Y)e^{\mathsf{T}}(\tilde{E}|_{p_i})}{e^{\mathsf{T}}(T_{p_i}Y)e^{\mathsf{T}}(\mathbb{E}^*\ot \tilde{E}|_{p_i})} \, ,$$ 
where $\mathbb{E}$ is the Hodge bundle on the moduli stack $\overline{M}_{1,1}$ of stable one pointed genus $1$ curves.

We recall the notations in Section \ref{No} for the series $\mathbf{R}_{ik}, \mathbf{U}_{i}, \mathbf{V}_{ii}$. The following theorem was proven in \cite{KL}.

\begin{Thm}\label{G1}
Let $\mathsf{F}_1(q_1,q_2)$ be the genus one quasimap potential function of $X$. For $k=1,2$ we have
 \begin{multline*} 
 q_k\frac{\partial}{\partial q_k} \mathsf{F}_1 =\sum_i q_k\frac{\partial}{\partial q_k}( - \text{log} \frac{\mathsf{R}_{i0}}{24}+c_i(\lambda)\frac{\mathsf{U}_i}{24} )\\
 +\frac{1}{2} \sum_i ({H_k}|_{p_i}+q_k\frac{\partial}{\partial q_k}\mathsf{U}_i) \text{lim}_{(x,y)\rightarrow (0,0)}\left(e^{-\mathsf{U}_i(\frac{1}{x}+\frac{1}{y})} e_i \mathsf{V}_{ii}(x,y)-\frac{1}{x+y}  \right) \,.
 \end{multline*}
\end{Thm}

\subsection{Calculations.}
We apply Theorem \ref{G1} to the hypersurface of degree $(2,n)$ in $\PP^1 \times \PP^{n-1}$. 
 
Let $X$ be the hypersurface of degree $(2,n)$ in  $\PP^1 \times \PP^{n-1}$. $X$ has curve class $(d_1,d_2)\in\NN \times\NN$ where $d_1$(resp. $d_2$) is degree inside $\PP^1$(resp. $\PP^{n-1}$). In this sections, we only consider curves class $(d_1,0)$.

\begin{Thm}\label{Th3} Let $\mathsf{F}_1$ be the genus one quasimap potential function of $X$. Then we have
 $$q\frac{d}{dq}\mathsf{F}_1(q,0)=\frac{n(n^2-n+2)}{2}(-\frac{1}{6}\frac{1}{1-4q}).$$
\end{Thm}

\subsubsection{$\mathsf{T}$-equivariant thoery.}
Let $\mathsf{T}:=(\CC^*)^2\ti (\CC^*)^n$ act on $\PP^1\ti\PP^{n-1}$ standardly on each components. Denote by
$\alpha_1,\alpha_2$ (resp. $\lambda_1,\dots,\lambda_{n}$) the weights of $(\CC^*)^2$ (resp. $(\CC^*)^n$). There are $2n$ fixed loci
$$p_i \in (\PP^1\ti\PP^{n-1})^{\mathsf{T}}\,\, \text{for}\, i=-n,\dots,-1,1,\dots,n$$
with following conditions:
\begin{align*}
 &H_1|_{p_i} = \left\{ \begin{array}{rl} \alpha_1 & \text{if } i > 0   \\
                                    \alpha_2 & \text{if } i < 0  \, , \end{array}\right. \\
 &H_2|_{p_i}=\lambda_{i}.
\end{align*}
Here we used the convention
\begin{align*}
\lambda_{-i}=\lambda_i \,\, \text{for} \, i=1,\dots,n.
\end{align*}

Also we will use following specialization,
$$\alpha_1=1\,\,,\,\,\alpha_2=-1\,.$$

\subsubsection{Picard-Fuchs equation.}
As in \cite{BigI}, we can evaluate the $\mathbf{I}$-function of Definition \ref{BigI} for $X$.

\begin{Prop}
 For $\mathbf{t}=t_1 \tilde{H}_1+t_2 \tilde{H}_2\in H^*_{\mathsf{T}}([\CC^2\ti\CC^n/\CC^*\ti\CC^*],\QQ)$,
 \begin{multline*}
     \mathbf{I}(t_1,t_2)=\sum_{d_1,d_2=0}^{\infty} q_1^{d_1}q_2^{d_2}e^{\sum_{i=1}^2 t_i(H_i+d_iz)/z}\\
     \cdot \frac{\prod_{k=1}^{2d_1+nd_2}(2H_1+nH_2+kz)}{\prod_{k_1}^{d_1}(H_1+1+k_1z)\prod_{k_1}^{d_1}(H_1-1+k_1z)\prod_{i=1}^{n}\prod_{k_2}^{d_2}(H_2-\lambda_i+k_2z)}\,.
 \end{multline*}
\end{Prop}
The function $\mathbf{I}$ satisfies following Picard-Fuchs equation

\begin{align}\label{PF2}
\left(\left(z\frac{\partial}{\partial t_1}\right)^2-1-q_1\prod_{k=1}^2\left(2 z\frac{\partial}{\partial t_1}+n z\frac{\partial}{\partial t_2}+kz\right)\right)\mathbf{I}=0.
\end{align}

Define $\mathbf{L}_i,\mathbf{a}_i$ and $\mathbf{b}_i$ by following equation:
\begin{align}\label{AF}
\mathbf{I}_i&=\exp\left(\frac{\mathbf{U}_i+\mathbf{a}_i z+\mathbf{b}_i z}z^2\right)+\mathcal{O}(z^2)\,,\\
\nonumber\mathbf{L}_i&=\frac{\partial}{\partial t_1}\mathbf{U}_i\,.
\end{align}

Applying \eqref{AF} to \eqref{PF2}, we obtain following equations.
$$(1-4q)L_i^2-4 n q \lambda_i L_i-n^2 q \lambda_i^2-1=0\, ,$$
\begin{align*}
    q\frac{ \partial}{\partial q}a_i&=\frac{-q\frac{ \partial}{\partial q}L_i (1-4q)+ 6 q L_i+3 n q \lambda_i}{2L_i(1-4q)-4nq\lambda_i}\\
                                                &=\frac{q(-8+\lambda_i^2 n^2(1-8q)+32q-2\lambda_i n\sqrt{1+(-4+\lambda_i^2 n^2)q})}{4(-1+4q)(1+(-4+\lambda_i^2 n^2)q)}\, ,\\
    q\frac{ \partial}{\partial q}b_i&=\frac{(-(q\frac{ \partial}{\partial q}a_i)^2-(q\frac{ \partial}{\partial q})^2 a_i)(1-4q)+6q q\frac{ \partial}{\partial q}a_i+2q}{2L_i(1-4q)-4nq \lambda_i}\\
                                              &=-\frac{\lambda_i n q(8-4 \lambda_i n\sqrt{1-4q+\lambda_i^2 n^2 q}+(-4+\lambda_i^2 n^2)q(8+\lambda_i n\sqrt{1-4q+\lambda_i^2 n^2 q})  )}{32(1+(-4+\lambda_i^2 n^2)q)^3}
\end{align*}
Here we used the notation in Section \ref{No} for the series $\mathbf{I}_i$, $\mathbf{L}$, $\mathbf{a}_i$ and $\mathbf{b}_i$.

By solving above differential equations with initial conditions
$$L_i(0)=1\,,\,a_i(0)=0\,, \, b_i(0)=0\,,\,\, \text{for} \,\, i=1,\dots,n\,\,,$$
we obtain following results for $i=1,\dots,n$.
\begin{align*}
  L_i&=\frac{2 \lambda_i n q + \sqrt{1-4q+\lambda_i^2 n^2 q}}{1-4q}\,,\\
  a_i&=\frac{1}{4}(-2 \text{ArcTanh}[\frac{2}{\lambda_i n}]+2\text{ArcTanh}[\frac{2\sqrt{1-4q+\lambda_i^2 n^2 q}}{\lambda_i n}]-\text{Log}[(1-4q)(1-4q+\lambda_i^2 n^2 q)])\,,\\
  b_i&=\frac{\lambda_i n(-6+\lambda_i n)}{24(-4+\lambda_i^2 n^2)}-\frac{\lambda_i n(-6\lambda_i^3 n^3 q+4 \lambda_i(n+6nq)-24\sqrt{1+(-4+\lambda_i^2 n^2)q})}{96(-4+\lambda_i^2 n^2)(1+(-4+\lambda_i^2 n^2)q)^{\frac{3}{2}}}\,.
\end{align*} 
 
Similiarly, solving above differential equations with initial conditions
$$ L_i(0)=-1\,,\,a_i(0)=0\, ,\, b_i(0)=0\, ,\,\, \text{for}\, i=-n,\dots,-1\, ,$$
we obtain following results for $i=-n,\dots,-1$. 
 \begin{align*}
  L_i&=\frac{2 \lambda_i n q - \sqrt{1-4q+\lambda_i^2 n^2 q}}{1-4q}\,,\\
  a_i&=\frac{1}{4}(2 \text{ArcTanh}[\frac{2}{\lambda_i n}]-2\text{ArcTanh}[\frac{2\sqrt{1-4q+\lambda_i^2 n^2 q}}{\lambda_i n}]-\text{Log}[(1-4q)(1-4q+\lambda_i^2 n^2 q)])\,,\\
  b_i&=-\frac{\lambda_i n(6+\lambda_i n)}{24(-4+\lambda_i^2 n^2)}+\frac{\lambda_i n(-6\lambda^3 n^3 q+4 \lambda_i(n+6nq)+24\sqrt{1+(-4+\lambda_i^2 n^2)q})}{96(-4+\lambda_i^2 n^2)(1+(-4+\lambda_i^2 n^2)q)^{\frac{3}{2}}}\,.
\end{align*}

\subsubsection{$\mathbf{S}$-operators.}

We introduce some power series in $q$. Define $I_0$, $I_1$ and $\tilde{I}_1$ by the following equation.
$$\mathbf{I}|_{t_1=0,t_2=0,q_2=0,q_1=q}=I_0+\frac{I_1 H_1+\tilde{I}_1 H_2}{z}+ \mathcal{O}(\frac{1}{z^2}).$$ 

We can explicitly calculate each series.

\begin{align*}
    &I_0=\sum \frac{(2d)!}{(d!)^2} q^d\,,\\
    &I_1=2 \sum \frac{(2d)!}{(d!)^2}(\text{Har}[2d]-\text{Har}[d])q^d\,,\\
    &\tilde{I}_1=n\sum \frac{(2d)!}{(d!)^2}\text{Har}[2d]q^d\,.
\end{align*}
Here $\text{Har}[n]:=\sum_{k =1}^n \frac{1}{k}$ .
 
It is easy to check following identities.
\begin{align*}
    &I_0=\frac{1}{\sqrt{1-4q}}\, ,\\
    &1+q\frac{d}{dq}(\frac{I_1}{I_0})=\frac{1}{\sqrt{1-4q}}\, ,\\
    &q\frac{d}{dq}\left(\frac{\tilde{I}_1}{I_0}\right)=n\left(\frac{4q}{1-4q}-\frac{1}{2\sqrt{1-4q}}+\frac{1}{2}\right)\,.
\end{align*} 

Define $\mathcal{Y}=q\frac{d}{dq}(\frac{\tilde{I}_1}{I_0})$.
Using Birkhoff factorization procedure, we obtain following equations for $\mathbf{S}$-operators.
\begin{align*}
 \mathbf{S}(1\ot 1)&=\frac{\mathbf{I}}{I_0}\, , \\
 \mathbf{S}(H_1\ot 1)&=\frac{z\frac{d}{dt_1}\mathbf{S}(1\ot 1)-\mathcal{Y} z\frac{d}{dt_2}\mathbf{S}(1\ot 1)}{I_0}\\
 &=\frac{1}{I_0}\left(z\frac{t}{dt_1}(\frac{\mathbf{I}}{I_0})- \mathcal{Y} H_2 \frac{\mathbf{I}}{I_0}\right)\,,\\
 \mathbf{S}(1\ot H_2^k)&=H_2^k \mathbf{S}(1\ot 1) \,,\\
 \mathbf{S}(H_1\ot H_2^k)&=H_2^k \mathbf{S}(H_1\ot 1)\,.
\end{align*}

Applying asymtotic form \eqref{AF} to above equation, we obtain the following results.

\begin{align}\label{Ass}
 e^{-\frac{U_i}{z}}S_i(1\ot H_2^j)&=\frac{\lambda_i^je^{a_i}}{I_0}+\frac{\lambda_i^je^{a_i} b_i}{I_0}z+\mathcal{O}(z^2)\, , \\
 \nonumber e^{-\frac{U_i}{z}}S_i(H_1\ot H_2^j)&=\frac{\lambda_i^je^{a_i}}{I_0}\left(\frac{L}{I_0}-\frac{\mathcal{Y} H_2}{I_0}\right)\\
 \nonumber &+\frac{\lambda_i^je^{a_i}}{I_0}\left(\frac{L b_i}{I_0}-\frac{\mathcal{Y} H_2 b_i}{I_0}+\frac{q\frac{\partial}{\partial q} a_i}{I_0}-\frac{q\frac{\partial}{\partial q}I_0}{I_0^2}\right)z+\mathcal{O}(z^2),\ .
\end{align}
Here we also used the notations in Section \ref{No} for the series $\mathbf{S}_i$.

\subsubsection{Proof of Theorem \ref{Th3}}

Since $\mathsf{F}_1$ do not depend on $\boldsymbol{\lambda}:=(\lambda_0,\dots,\lambda_{n-1})$ by dimensional reason, we have 

$$q\frac{\partial}{\partial q} \mathsf{F}_1(q,0)=\text{lim}_{\boldsymbol{\lambda} \rightarrow 0}q\frac{\partial}{\partial q} \mathsf{F}_1(q,0) .$$
Throughout the paper, whenever we take the limit $$\boldsymbol{\lambda}\rightarrow 0\,,$$
we can check that it is well-defined.
Denote by Vert (resp. Loop) the first (resp. the second) summand in Theorem \ref{G1}. We can calculate each term of Vert as follows.

\begin{align*}
\text{lim}_{\boldsymbol{\lambda}\rightarrow 0} \sum_{i} q \frac{\partial }{\partial q}a_i = 2n \frac{2q}{1-4q}.
\end{align*}

\begin{align*}
&\text{lim}_{\boldsymbol{\lambda}\rightarrow 0} \sum_i c_i(\lambda) L_i\\
&=\text{lim}_{\boldsymbol{\lambda}\rightarrow 0}\sum_{i=1}^n \left(-\frac{1}{2}-(\sum_{j\ne i}\frac{1}{\lambda_i-\lambda_j})+\frac{1}{2+n \lambda_i}\right)\left(\frac{2 \lambda_i n q+\sqrt{1-4q+\lambda_i^2 n^2 q}}{1-4q}-1\right)\\
&+\text{lim}_{\boldsymbol{\lambda}\rightarrow 0}\sum_{i=1}^n \left(\frac{1}{2}-(\sum_{j\ne i}\frac{1}{\lambda_i-\lambda_j})+\frac{1}{-2+n \lambda_i}\right)\left(\frac{2 \lambda_i n q-\sqrt{1-4q+\lambda_i^2 n^2 q}}{1-4q}+1\right)\\
&=\frac{-4n^2(n-1)}{2}\frac{q}{1-4q}.
\end{align*}

Therefore we have
\begin{align*}
\text{lim}_{\boldsymbol{\lambda}\rightarrow 0} \text{Vert}= -\frac{n(n^2-n+2)}{12}\frac{q}{1-4q}.
\end{align*}

Now to finish the proof, we need to show
$$\text{lim}_{\boldsymbol{\lambda}\rightarrow 0} \text{Loop}=0\,.$$

\begin{align*}
\text{Loop}&=\frac{1}{2} \sum_i (H_1 |_{p_i}+q\frac{\partial}{\partial q}U_i) \text{lim}_{(x,y)\rightarrow (0,0)}\left(e^{-U_i(\frac{1}{x}+\frac{1}{y})} e_iV_{ii}(x,y)-\frac{1}{x+y}  \right) \\
&=\frac{1}{2}\sum_i \frac{L_i}{e_i}\left[e^{-U_i(\frac{1}{x}+\frac{1}{y})}\sum_i S_i(\phi_k) S_i(\phi^k)\right]_x\\
&=\frac{1}{2}\sum_i \frac{L_i}{e_i}\left(\frac{-2\prod_{j \ne i}(\lambda_i-\lambda_j)}{(-2+n\lambda_i)(2+n\lambda_i)}(A_i+C_i)+\frac{n\lambda_i\prod_{j \ne i}(\lambda_i-\lambda_j)}{(-2+n\lambda_i)(2+n\lambda_i)}B_i\right)
\end{align*}
where

\begin{align*}
A_i&=\left[e^{-U_i(\frac{1}{x}+\frac{1}{y})}S_i(1) S_i(1)\right]_x\,,\\
B_i&=\left[e^{-U_i(\frac{1}{x}+\frac{1}{y})}S_i(1) S_i(H_1)+e^{-U_i(\frac{1}{x}+\frac{1}{y})}S_i(H) S_i(1)\right]_x\,,\\
C_i&=\left[e^{-U_i(\frac{1}{x}+\frac{1}{y})}S_i(H_1) S_i(H_1)\right]_x\,.
\end{align*}

By \eqref{Ass}, we obtain the following results.
\begin{align*}
A_i&=e^{2a_i} \frac{b_i}{I_0^2}\,, \\
B_i&=\frac{e^{2a_i}}{I_0^2}\left(2\left(\frac{L_i b_i}{I_0}-\frac{\mathcal{Y} b_i \lambda_i}{I_0}\right)+\frac{q\frac{\partial}{\partial q}a_i}{I_0}-\frac{q\frac{\partial}{\partial q}I_0}{I_0^2}\right)\,,\\
C_i&=\frac{e^{2a_i}}{I_0^2}\left(\frac{L_i}{I_0}-\frac{\mathcal{Y}\lambda_i}{I_0}\right)\left(\frac{L_i b_i}{I_0}-\frac{\mathcal{Y} b_i \lambda_i}{I_0}+\frac{q\frac{\partial}{\partial q}a_i}{I_0}-\frac{q\frac{\partial}{\partial q}I_0}{I_0^2}\right)\,.
\end{align*}

We can easily check the followings.
\begin{align*}
 &\text{lim}_{\boldsymbol{\lambda}\rightarrow 0} b_i=0\\
 &\text{lim}_{\boldsymbol{\lambda}\rightarrow 0} q\frac{\partial}{\partial q} a_i = \frac{q\frac{\partial}{\partial q} I_0}{I_0}
\end{align*}

As a results, we conclude
\begin{align*}
&\text{lim}_{\boldsymbol{\lambda}\rightarrow 0} A_i=0\\
&\text{lim}_{\boldsymbol{\lambda}\rightarrow 0} B_i=0\\
&\text{lim}_{\boldsymbol{\lambda}\rightarrow 0} C_i=0.
\end{align*}

Therefore we have
$$\text{lim}_{\boldsymbol{\lambda}\rightarrow 0} \text{Loop}=0$$
which finish the proof.\qed

\subsection{General case.} In this section we study the genus one quasimap invariants of hypersurface of degree $(m,n)$ in $\PP^{m-1}\times\PP^{n-1}$. Let X be the hypersurface of degree $(m,n)$ in $\PP^{m-1}\times\PP^{n-1}$. As before, we denote by $$\mathsf{F}_1$$ the genus one quasimap potential function of X. We will give a poof of Theorem \ref{Th4} in this section.

\subsubsection{$\mathsf{T}$-equivariant theory.} Let $(\CC^*)^m\times(\CC^*)^n$ act on $\PP^{m-1}\times\PP^{n-1}$ standardly on each components. Denote by $\alpha_0,\alpha_1,\dots\alpha_{m-1}\,\,(\text{resp.}\, \lambda_0,\dots,\lambda_{n-1})$ the weight of $(\CC^*)^m$  (resp. $(\CC^*)^n$). There are $m\times n$ fixed points
\begin{align*}
 p_{ki}\,\,\, \text{for}\,\,\, k=0,\dots,m-1\,\,,\,i=0,\dots,n-1\,
\end{align*}
with following conditions:
\begin{align*}
    &H_1|_{p_{ki}}=\alpha_k\,,\\
    &H_2|_{p_{ki}}=\lambda_i\,.
\end{align*}

We will use the following specialization throught the subsection,
$$\alpha_i=\zeta^i.$$
Here, $\zeta$ is primitive $m$-th root of unity.

\subsubsection{$\mathbf{I}$-function and Picard-Fuchs equation} As before, we can evaluate the $\mathbf{I}$-function for $X$.

\begin{Prop}
 For $\mathbf{t}=t_1 \tilde{H}_1+t_2 \tilde{H}_2\in H^*_{\mathsf{T}}([\CC^m\ti\CC^n/\CC^*\ti\CC^*],\QQ)$,
 \begin{multline*}
     \mathbf{I}(t_1,t_2)=\sum_{d_1,d_2=0}^{\infty} q_1^{d_1}q_2^{d_2}e^{\sum_{i=1}^2 t_i(H_i+d_iz)/z}\\
     \cdot \frac{\prod_{k=1}^{md_1+nd_2}(mH_1+nH_2+kz)}{\prod_{j=0}^{m-1}\prod_{k_1}^{d_1}(H_1-\alpha_j+k_1z)\prod_{i=0}^{n-1}\prod_{k_2}^{d_2}(H_2-\lambda_i+k_2z)}\,.
 \end{multline*}
\end{Prop}

The function $\mathbf{I}$ satisfies following Picard-Fuchs equation,

\begin{align*}
    \left(\left(z\frac{\partial}{\partial t_1}\right)^2-1-q_1\prod_{l=1}^{m}\left(m z\frac{\partial}{\partial t_1}+n z\frac{\partial}{\partial t_2}+lz\right)\right)\mathbf{I}=0
\end{align*}

Define $\mathbf{L}_{ki}, \mathbf{a}_{ki}$ and $\mathbf{b}_{ki}$ by following identity.
\begin{align*}
\mathbf{I}|_{p_{ki}}=&e^{\frac{\mathbf{U}_{ki}+\mathbf{a}_{ki}z+\mathbf{b}_{ki}z^2}{z}}\,,\\
\mathbf{L}_{ki}=&\frac{d}{dt_1}\mathbf{U}_{ki}\,.
\end{align*}

We obtain the following results.
\begin{align}\label{BE}
    \nonumber\sum_k L_{ki}&=\frac{n\lambda_i m^mq}{1-m^mq}\,,\\
    \text{lim}_{\boldsymbol{\lambda}\rightarrow 0} L_{ki}&=\alpha^k L\,,\\
    \nonumber\text{lim}_{\boldsymbol{\lambda}\rightarrow 0} a_{ki}&=\text{Log}L\,,\\
    \nonumber\text{lim}_{\boldsymbol{\lambda}\rightarrow 0} b_{ki}&=\frac{(m-2)(m+1)}{24m\alpha^k}(1-L^n)
\end{align}
where $L=(1-m^mq)^{-\frac{1}{m}}$.

\subsubsection{$\mathbf{S}$-operators.}
Since we will use the specializations $\lambda_i=0$ in the end, it is enough for us to calculate $\mathbf{S}$-operator modulo $\lambda_i$.

\begin{align}\label{S2}
    \mathbf{S}(H_1^k)=\frac{z\frac{d}{dt_1}\mathbf{S}(H_1^{k-1})}{C_k}+\mathcal{O}(\boldsymbol{\lambda},q_2)\,,\\
    \nonumber\mathbf{S}(H_1^k\ot H_2^l)=H_2^l \mathbf{S}(H_1^k)+\mathcal{O}(\boldsymbol{\lambda},q_2)\,.
\end{align}
where $C_k(q_1)\in\CC[[q_1]]$ is defined inductively by
\begin{align}\label{Ckg}
    C_k=z\frac{d}{dt_1}\mathbf{S}(H_1^{k-1})|_{z=\infty,t_1=0,t_2=0,q_2=0,H_1=1}\,,\,\,\,\text{for}\,\,k=1,2,\dots,m-1\,,
\end{align}
with $C_0(q_1):=\mathbf{S}(1)|_{z=\infty,t_1=0,t_2=0,q_2=0}$.
\subsubsection{Proof of Theorem \ref{Th4}.}

As in the previous subsection, we have 
\begin{align*}
 q\frac{\partial}{\partial q}F_1=\text{lim}_{\boldsymbol{\lambda}\rightarrow 0} q\frac{\partial}{\partial q}F_1.
\end{align*}

The followings are immidiate consequences of \eqref{BE}.
\begin{align*}
    \text{lim}_{\boldsymbol{\lambda}\rightarrow 0}\sum_{k,i} q\frac{\partial}{\partial q}a_{ki}&=mn\frac{\mathsf{D}L}{L}\,,\\
    \text{lim}_{\boldsymbol{\lambda}\rightarrow 0}\sum_{k,i} c_{ki} L_{ki}&=-\frac{(m^2-m-2)nL}{2}-\frac{(n-1)n^2}{2}\frac{m^mq}{1-m^mq}\,.
\end{align*}

Therefore we obtain the following results.
\begin{align}\label{Vert}
    \text{lim}_{\boldsymbol{\lambda}\rightarrow 0}\text{Vert}=\frac{1}{48}n(-n^2+n-2)(L^m-1)+\frac{1}{48}n(2+m-m^2)(L-1)\,.
\end{align}

Now we calculate the loop contribution. We have 

\begin{align*}
    e_{ki}=\frac{\alpha^k-\alpha^{k+1}\dots(\alpha^k-\alpha^{k+m-1})\prod_{j\ne i}(\lambda_i-\lambda_j)}{m\alpha^k+n\lambda_i}\,,\\
    \phi_{ki}=\frac{(H_1-\alpha^{k+1})\dots(H_1-\alpha^{k+m-1})}{(\alpha^k-\alpha^{k+1})\dots(\alpha^k-\alpha^{k+m-1})}\prod_{j\ne i}\frac{H_2-\lambda_j}{\lambda_i-\lambda_j}\, ,\\
    \phi^{ki}=\frac{(H_1-\alpha^k)\dots(H_1-\alpha^{k+m-1})\prod_{j\ne i}(H_2-\lambda_j)}{m\alpha^k+n\lambda_i}\,.
\end{align*}

We obtain following calculation of loop contributions for each fixed point $p_{ki}$.
\begin{align*}
    &\text{lim}_{\boldsymbol{\lambda}\rightarrow 0}\frac{1}{2}\frac{L_{ki}}{e_{ki}}\sum_{0\le j_1 \le m-1,0\le j_2\le n-1}S_{ki}(\phi_{j_1j_2})S_{ki}(\phi^{j_1j_2})\\
    =&\text{lim}_{\boldsymbol{\lambda}\rightarrow 0}\frac{1}{2}\frac{L_{ki}}{e_{ki}}\sum_{0\le r \le m-1}S_{ki}(\phi_{ri})S_{ki}(\phi^{ri})\\
    =&\text{lim}_{\boldsymbol{\lambda}\rightarrow 0}\frac{1}{2}\sum_{0\le r \le m-1}\frac{L_{ki}(m\alpha^k+n\lambda_i)}{(\alpha^k-\alpha^{k+1})\dots(\alpha^k-\alpha^{k+m-1})\prod_{j\ne i}(\lambda_i-\lambda_j)}S_{ki}(\tilde{\phi}_r)S_{ki}(\tilde{\phi}^r)\\
    \cdot&\frac{\prod_{j\ne i}(\lambda_i-\lambda_j)\cdot m\alpha^k}{m\alpha^k+n\lambda_i}\\
    =&\frac{L}{2\alpha^{k(m-3)}}\text{lim}_{\boldsymbol{\lambda}\rightarrow 0}\left[\frac{S_{ki}(1)S_{ki}(H_1^{m-2})+\dots+S_{ki}(H_1^{m-2})S_{ki}(1)+S_{ki}(H_1^{m-1})\ot S_{ki}(H_1^{m-1})}{m}\right]_{x}\\
    =&\frac{1}{2m}\left(\frac{(m^2-m-2)}{24}(L-1)+\frac{(3m-5)(m-2)}{24}(L^m-1)-\sum_{k=0}^{m-3}\binom{m-1-k}{2}\frac{q\frac{d}{dq}C_k}{C_k}\right)\,.
\end{align*}
where
\begin{align*}
    \tilde{\phi}_r&=\frac{(H_1-\alpha^{r+1})\dots(H_1-\alpha^{r+m-1})}{(\alpha^r-\alpha^{r-1})\dots(\alpha^r-\alpha^{r+m-1})}\,,\\
    \tilde{\phi}^r&=\frac{(H_1-\alpha^r)\dots(H_1-\alpha^{r+m-1})}{m\alpha^r}\,.
\end{align*}
The first equality in above equations follow from \eqref{S2}.
The last equality in above equations are also straightforward calculations using \eqref{S2}. Since similar calculations appeared in \cite{KL}, we do not repeat this in our paper. 

Finally we obtain the following result for the loop contribution.

\begin{multline}\label{Loop}
    \text{lim}_{\boldsymbol{\lambda}\rightarrow 0}\text{Loop}=\text{lim}_{\boldsymbol{\lambda}\rightarrow 0}\sum_{ki}\frac{1}{2}\frac{L_{ki}}{e_{ki}}\sum_{j_1,j_2}S(\phi_{j_1j_2})S(\phi^{j_1j_2})\\
    =\frac{n}{2}\left(\frac{(m^2-m-2)}{24}(L-1)+\frac{(3m-5)(m-2)}{24}(L^m-1)-\sum_{k=0}^{m-3}\binom{m-1-k}{2}\frac{q\frac{d}{dq}C_k}{C_k}\right)\,.
\end{multline}

By combining \eqref{Vert} and \eqref{Loop}, we obtain the equation in the Theorem \ref{Th4}. \qed


\begin{thebibliography}{99}



\bibitem{ASYZ} M. Alim, E. Scheidegger, S.-T. Yau, J. Zhou, {\em Special polynomial rings, quasi modular forms and duality of topological strings}, Adv. Theor. Mah. Phys. {\bf 18} (2014), 401--467.



\bibitem{BCOV} M. Bershadsky, S. Cecotti, H. Ooguri, and C. Vafa, {\em Holomorphic anomalies in topological field theories}, Nucl. Phys. B{\bf 405}
(1993), 279--304.



\bibitem{BKMP}
V. Bouchard, A. Klemm,
M. Mari\~no, and S. Pasquetti,
{\em Remodelling the B-model},
Comm. Math. Phys. {\bf 287}
(2009), 117--178.


\bibitem{CK} I. Ciocan-Fontanine and B. Kim, {\em Moduli stacks of stable toric quasimaps,}  Adv. in Math. {\bf 225} (2010), 3022--3051.

\bibitem {CKg0}  I. Ciocan-Fontanine and B. Kim,
{\em Wall-crossing in genus zero quasimap theory and mirror maps,}   Algebr. Geom. {\bf 1 } (2014), 400--448.

\bibitem {CKg}  I. Ciocan-Fontanine and B. Kim, 
{\em Higher genus quasimap wall-crossing for semi-positive targets}, JEMS {\bf 19}  (2017),
2051-2102.



\bibitem{BigI}   I. Ciocan-Fontanine and B. Kim, {\em Big I-functions} in
 {\em Development of moduli theory Kyoto 2013}, 323--347, Adv. Stud. Pure Math. {\bf 69}, 
 Math. Soc. Japan, 2016. 


\bibitem {CKw}  I. Ciocan-Fontanine and B. Kim, {\em Quasimap wallcrossings
and mirror symmetry}, arXiv:1611.05023.

\bibitem{CKM} I. Ciocan-Fontanine, B. Kim, and D. Maulik,  {\em Stable quasimaps to GIT quotients,}
  J. Geom. Phys. {\bf 75} (2014), 17--47.

\bibitem{CJR} E. Clader, F. Janda, and Y. Ruan,  {\em Higher genus quasimap wall-crossing
via localization},
  arXiv:1702.03427.




\bibitem{EMO} B.~Eynard,
M.~Mari\~no, and N.~Orantin,
{\em Holomorphic
anomaly and matrix
models}, JHEP {\bf 58} (2007).

\bibitem{FLZ}
B. Fang, M. C.-C. Liu,
and Z. Zong, {\em On the
remodelling conjecture
for toric Calabi-Yau 
3-orbifolds}, arXiv:1604.07123.



\bibitem{Gequiv} A. Givental, {\em Equivariant Gromov-Witten invariants}, Internat. Math. Res. Notices {\bf 13} (1996), 613--663.




\bibitem{Elliptic} A. Givental, {\em Elliptic Gromov-Witten invariants and the generalized mirror conjecture,}
 Integrable systems and algebraic geometry (Kobe/Kyoto, 1997), 107--155, World Sci. Publ., River Edge, NJ, 1998. 

\bibitem{SS} A. Givental, {\em Semisimple Frobenius structures
at higher genus}, Internat. Math. Res. Notices {\bf 23} (2001),  613--663.

\bibitem{GP}
T.~Graber, R.~Pandharipande,
\newblock {\em Localization of virtual classes}, Invent. Math. {\bf 135}
(1999),  487--518.

\bibitem{GJR} S. Guo, F. Janda, Y. Ruan, {\em A mirro theorem for genus two Gromov-Witten invariants of quintic threefolds}, arXiv:1709.07392.



\bibitem{KKP} B. Kim,  A. Kresch, T. Pantev, {\em Functoriality in intersection theory and a conjecture of Cox, Katz, and Lee, } J. Pure Appl. Algebra {\bf 179} (2003), 127--136.

\bibitem{KL} B. Kim and  H. Lho, {\em Mirror theorem for elliptic quasimap invariants}, Geom. Topol. {\bf 22} (2018), 1459--1481.


\bibitem{KP} A. Klemm and R. Pandharipande, {\em Enumerative geometry of Calabi-Yau 4-folds}, Comm. Math. Phys. {\bf 281} (2008), 621--653.








\bibitem{Book} Y.-P. Lee and R. Pandharipande, {\em Frobenius manifolds, Gromov-Witten theory and Virasoro constraints,} https://people.math.ethz.ch/\~{}rahul/, 2004.

\bibitem{LP1} H. Lho and R. Pandharipande,
{\em Stable quotients and the holomorphic anomaly equation},
arXiv:1702.06096.


\bibitem{MOP} A. Marian, D. Oprea, Dragos, R. Pandharipande, 
{\em The moduli space of stable quotients,} 
Geom. Topol. {\bf 15} (2011), 1651--1706.







\bibitem{PZ} R. Pandharipande and A. Zinger, {\em Enumerative geometry of Calabi-Yau 5-fodls}, in {\em New developments in algebraic geometry, integrable system, and mirror symmetry (RIMS, Kyoto 2008)}, Eds., Maht. Soc. of Japan (2010), 239--288.

\bibitem{Popa1} A. Popa 
{\em The genus one Gromov-Witten invariants of Calabi-Yau complete intersections,} 
Trans. Amer. Math. Soc. {\bf 365} (2013), 1149--1181. 

\bibitem{ZaZi} D. Zagier and A. Zinger, {\em Some properties of hypergeometric series associated with mirror symmetry} in {\em Modular Forms and String Duality}, 163-177, Fields Inst. Commun. {\bf 54}, AMS 2008.

\bibitem{Zg1}  A. Zinger,
{\em The reduced genus 1 Gromov-Witten invariants of Calabi-Yau hypersurfaces,}  J. Amer. Math. Soc. {\bf 22} (2009), 691--737.






\end{thebibliography}
\end{document}